\documentclass[11pt,a4paper]{amsart}

\usepackage[utf8]{inputenc}
\usepackage[T1]{fontenc}
\usepackage[english]{babel}
\usepackage{amsmath}
\usepackage{amsfonts}
\usepackage{ae}
\usepackage{units}
\usepackage{icomma}
\usepackage{color}
\usepackage{graphicx}
\usepackage{bbm}
\usepackage{amsthm}
\usepackage{amssymb}
\usepackage{cancel}
\usepackage{fullpage}
\usepackage{type1cm}
\usepackage{eso-pic}
\usepackage[breaklinks,pdfpagelabels=false]{hyperref}
\usepackage{xspace}
\usepackage{enumitem}
\usepackage{bbm}
\usepackage{verbatim}
\usepackage{subfig}
\usepackage{tikz}
\usetikzlibrary{arrows}
\usepackage{amsrefs}

\newcommand{\abs}[1]{\ensuremath{\left| #1 \right|}}

\newcommand{\bm}[1]{\mbox{\boldmath $ {#1} $}}

\newcommand{\vz}{\bm{\hat{0}}}
\newcommand{\vo}{\bm{\hat{1}}}
\newcommand{\var}{\operatorname{Var}}
\newcommand{\sech}{\operatorname{sech}}
\newcommand{\csch}{\operatorname{csch}}
\newcommand{\defeq}{:=}
\newcommand{\hc}[1]{\ensuremath{\mathbb{Q}_{ #1 }}}
\newcommand{\fpptime}{\vartheta}

\renewcommand{\epsilon}{\varepsilon}

\newtheorem{theorem}{Theorem}[section]
\newtheorem{lemma}[theorem]{Lemma}
\newtheorem{proposition}[theorem]{Proposition}
\newtheorem{corollary}[theorem]{Corollary}

\theoremstyle{definition}
\newtheorem{remark}[theorem]{Remark}

\renewcommand{\qed}{\hfill \mbox{\raggedright \rule{0.1in}{0.1in}}}

\numberwithin{equation}{section}
\numberwithin{theorem}{section}

\begin{document}
%\begin{titlepage}
\title{Unoriented first-passage percolation on the $n$-cube}
\author{Anders Martinsson}
\address{Department of Mathematical Sciences, Chalmers University Of Technology and University of Gothenburg, 41296 Gothenburg, Sweden}
%\address{Mathematical Sciences, Chalmers, 41296 Gothenburg, Sweden} 
%\address{Mathematical Sciences, University of Gothenburg,  41296 Gothenburg, Sweden}
\email{andemar@chalmers.se}
\date{\today}

\begin{abstract}
The $n$-dimensional binary hypercube is the graph whose vertices are the binary $n$-tuples $\{0, 1\}^n$ and where two vertices are connected by an edge if they differ at exactly one coordinate. We prove that if the edges are assigned independent mean 1 exponential costs, the minimum length $T_n$ of a path from $(0, 0, \dots, 0)$ to $(1, 1, \dots, 1)$ converges in probability to $\ln(1+\sqrt{2}) \approx 0.881$. It has previously been shown by Fill and Pemantle (1993) that this so-called first-passage time asymptotically almost surely satisfies $\ln(1+\sqrt{2}) - o(1) \leq T_n \leq 1+o(1)$, and has been conjectured to converge in probability by Bollobás and Kohayakawa (1997). A key idea of our proof is to consider a lower bound on Richardson's model, closely related to the branching process used in the article by Fill and Pemantle to obtain the bound $T_n \geq \ln\left(1+\sqrt{2}\right)-o(1)$. We derive an explicit lower bound on the probability that a vertex is infected at a given time. This result is formulated for a general graph and may be applicable in a more general setting.
\end{abstract}

\maketitle

\section{Introduction}
The $n$-dimensional binary hypercube $\hc{n}$ is the graph with vertex set $\{0, 1\}^n$ where two vertices share an edge if they differ at exactly one coordinate. We let $\vz$ and $\vo$ denote the all zeroes and all ones vertices respectively. For any vertex $v\in \hc{n}$, we let $\abs{v}$ denote the number of coordinates of $v$ that are $1$. A path $v_0 \rightarrow v_1 \rightarrow \dots \rightarrow v_k$ in $\hc{n}$ is called \emph{oriented} if $\abs{v_i}$ is strictly increasing along the path.

First-passage percolation is a random process on a graph $G$, which was introduced by Hammersley and Welsh. In this process, each edge $e$ in the graph is assigned a random variable $W_e$ called the \emph{passage time} of $e$. In this paper, the passage times will always be independent exponentially distributed random variables with expected value $1$. The usual way in which this process is described is that there exists some vertex $v_0 \in G$ which is assigned some property, usually either that it is infected ($v_0$ is the source of some disease) or wet ($v_0$ is connected to a water source), which then spreads throughout the graph. The passage time of an edge corresponds to the time it takes for an infection to spread in any direction along the edge, that is, when a vertex $v$ gets infected the infection spreads to each neighbor $w$ after $W_{\{v,w\}}$ time, assuming $w$ is not already infected at that time. More concretely, we can let the edge weights generate a metric on $G$. For a path $\gamma$ in $G$ we define the \emph{passage time} of $\gamma$ as the sum of passage times of the edges along $\gamma$. Moreover, for any two vertices $v, w \in G$, we say that the \emph{first-passage time} from $v$ to $w$, denoted by $d_{W}(v, w)$, is the infimum of passage times over all paths from $v$ to $w$ in $G$. Then for any $v \in G$, the time at which $v$ is infected is given by $d_W(v_0, v)$.

An alternative way to formulate first-passage percolation with independent exponentially distributed passage times is to consider the process $\{R(\cdot, t)\}_{t\geq 0}$, where for each $t\geq 0$, $R(v, t)$ is the map from the vertex set of $G$ to $\{0, 1\}$ given by
\begin{equation}
R(v, t) = \begin{cases} 1 &\mbox{if } d_W(v_0, v) \leq t\\ 0 &\mbox{otherwise,}\end{cases}
\end{equation}
that is, $R(v, t)$ is the indicator function for the event that $v$ is infected at time $t$. When the edge passage times are independent exponentially distributed with mean one, the memory-less property implies that the process $\{R(\cdot, t)\}_{t\geq 0}$ is Markovian, and its distribution is given by the initial condition $R(\cdot, 0) = \delta_{v_0, \cdot}$ together with the transitions $\{R(\cdot) \rightarrow R(\cdot) + \delta_{v,\cdot}\}$ at rate equal to the number of infected neighbors of $v$ if $v$ is healthy, and $0$ if $v$ is infected, see \cite{d88}. Here $\delta_{\cdot,\cdot}$ denotes the Kronecker delta function. This Markov process is known as Richardson's model.

First-passage percolation and Richardson's model on the hypercube have previously been studied by Fill and Pemantle \cite{fp93}, and later by Bollob\'{a}s and Kohayakawa \cite{bk97}. For Richardson's model we always assume that the original infected vertex is $\vz$, though by transitivity of the hypercube it is clear that the analogous statements hold for any starting vertex. The quantities considered in these articles of most relevance to this paper are the first-passage time from $\vz$ to $\vo$, which we denote by $T_n$, the \emph{oriented first-passage time} from $\vz$ to $\vo$, and the \emph{covering time}. Note that, in terms of Richardson's model, $T_n$ is the time until the vertex furthest from the starting point gets infected. The oriented first-passage time is a simplified version of the first-passage time, first proposed by Aldous \cite{a89}, where the minimum is only taken over all oriented paths from $\vz$ to $\vo$. The covering time is the random amount of time in Richardson's model on $\hc{n}$ until all vertices are infected or, equivalently, $\max_{v\in \hc{n}} d_W(\vz, v)$, the maximum first-passage time from $\vz$ to any other vertex in $\hc{n}$.

In case of oriented first-passage percolation, it was shown by Fill and Pemantle that the oriented first-passage time from $\vz$ to $\vo$ converges to $1$ in probability as $n\rightarrow \infty$. The fact that $1-o(1)$ is an asymptotic almost sure lower bound had already been observed by Aldous in \cite{a89}, and can be shown in a straight-forward manner by considering the expected number of oriented paths from $\vz$ to $\vo$ with passage time at most $t$. The argument by Fill and Pemantle for the upper bound is  essentially a second moment analysis on the number of such paths, though as they remark, a direct application of the second moment method can only show that the probability that the oriented first-passage time is at most $1+\varepsilon$ is bounded away from $0$. To circumvent this, they consider a ``variance reduction trick'', which effectively means that they consider a slightly different random variable.

For the unoriented first-passage time from $\vz$ to $\vo$, Fill and Pemantle showed that, as $n\rightarrow \infty$, we have
\begin{equation}
\ln\left(1+\sqrt{2}\right) - o(1) \leq T_n \leq 1+o(1)
\end{equation}
with probability $1-o(1)$. The upper bound follows directly from the oriented first-passage time. They remark that they doubt the upper bound is sharp, but state that they do not know how to improve it. Prior to this article, this seems to be the best known upper bound on $T_n$. For the lower bound, Fill and Pemantle relayed an argument by Durrett. In this argument we consider a random process on $\hc{n}$, which Durrett calls a branching translation process (BTP). We will postpone the definition of this process to the next section, but the essential difference to Richardson's model is that we allow each site to contain multiple instances of the infection at the same time. Durrett argues that this process stochastically dominates Richardson's model in the sense that it is possible to couple the models such that the infected vertices in Richardson's model are always a subset of the so-called occupied vertices in the BTP. He proves that the time at which $\vo$ becomes occupied tends to $\ln\left(1+\sqrt{2}\right)$ in probability as $n\rightarrow \infty$. As BTP stochastically dominates Richardson's model, this directly implies that $T_n \geq \ln\left(1+\sqrt{2}\right)-o(1)=0.881\dots-o(1)$ with probability $1-o(1)$.

Bollob\'{a}s and Kohayakawa \cite{bk97} showed that many global first-passage percolation properties on $\hc{n}$, such as the covering time and the graph diameter with respect to $d_W(\cdot, \cdot)$, can be bounded from above in terms of $T_n$. They defined the quantity
\begin{equation}
T_\infty = \inf \left\{t \in \mathbb{R}\middle\vert \mathbb{P} \left(T_n \leq t\right)\rightarrow 1 \text{ as }n\rightarrow \infty \right\}.
\end{equation}
Their main result is that asymptotically almost surely the covering time is at most $T_\infty + \ln 2 + o(1)$ and the graph diameter is at most $T_\infty + 2 \ln 2 + o(1)$. Note that it follows from the results by Fill and Pemantle that $\ln\left(1+\sqrt{2}\right) \leq T_\infty \leq 1$. Furthermore, it is easy to see that if $T_n$ converges in probability as $n\rightarrow \infty$, then it must converge to $T_\infty$. In fact, Bollob\'{a}s and Kohayakawa explicitly conjectured that this is the case, and they consequently referred to $T_\infty$ as simply the first-passage percolation time between two antipodal vertices in $\hc{n}$. While their article does not prove that $T_n$ converges in probability, the ideas do have some implications for $T_n$. For instance, with some small modifications of their proof it follows that if $T_n$ converges in distribution, then the limit must be concentrated on one point, meaning that $T_n$ converges in probability.

Besides first-passage percolation, percolation on the hypercube with restriction to oriented paths has also been considered in regards to Bernoulli percolation by Fill and Pemantle (in the same article), and, more recently, accessibility percolation\footnote{The name accessibility percolation is not mentioned in the cited article. The term was coined by Joachim Krug and Stefan Nowak after its writing.} by Hegarty and the author in \cite{hm13}. Common for these three cases of oriented percolation is that the proofs are based on second moment analyses. Arguably, this is made possible by the relatively simple combinatorial properties of oriented paths. We have $n!$ oriented paths from $\vz$ to $\vo$ in $\hc{n}$, all of length $n$ and all equivalent up to permutation of coordinates. Perhaps more importantly, one can derive good estimates on the number of pairs of oriented paths from $\vz$ to $\vo$ that intersect a given number of times, something which is made possible by the natural representation of oriented paths as permutations. In contrast, general paths from $\vz$ to $\vo$ do not seem to have a similar representation in any meaningful way, and in any case, there is certainly a lot more variation between general paths than oriented such. Hence, it seems that these type of ideas from oriented percolation on the hypercube cannot be transferred to unoriented percolation.

The most promising approach to improve the result by Fill and Pemantle for $T_n$ seems to be the BTP. Comparing the BTP to path-counting arguments, on the hypercube the former has the advantage that a number of relevant quantities, such as moment estimates, can be expressed by explicit analytical expressions, hence circumventing the problem of counting paths. However, beyond the fact that the BTP stochastically dominates Richardson's model, the relation between the two models is fairly subtle. It is therefore not immediately clear how proving anything about the BTP could imply upper bounds on the first-passage time.

In this article, we propose a way to do precisely this. A central idea of our approach is to consider a subprocess of the BTP with two important properties: Firstly, Richardson's model is stochastically sandwiched between the full BTP and this subprocess, and secondly, it is possible to derive an explicit lower bound on the probability that a vertex is occupied at a given time in this subprocess, expressed in tractable quantities for the BTP. Applying these ideas to the hypercube, we are able to resolve the problem of determining the limit of $T_n$. This result is summarized in the following Theorem, which is the main result of this paper:
\begin{theorem}\label{thm:unorientedFPP}
Let $T_n$ denote the first-passage time from $\vz$ to $\vo$ in $\hc{n}$ with exponentially distributed edge costs with mean $1$. For any $1\leq p < \infty$ we have $\|T_n-\ln\left(1+\sqrt{2}\right)\|_p=\Theta\left(\frac{1}{n}\right)$. In particular, we have $\mathbb{E}T_n = \ln\left(1+\sqrt{2}\right)+O\left(\frac{1}{n}\right)$ and $\var\left(T_n\right)=\Theta\left(\frac{1}{n^2}\right)$.
\end{theorem}

A direct consequence of this result is that $T_\infty = \ln\left(1+\sqrt{2}\right)$, which in particular improves the best known upper bound on the covering time to  $\ln\left(1+\sqrt{2}\right) + \ln 2 + o(1) = 1.574\dots + o(1)$. One can compare this with the best known lower bound  $\frac{1}{2} \ln \left(2+\sqrt{5}\right)+\ln 2 - o(1)=1.414\dots - o(1)$, as shown by Fill and Pemantle.

Given this result for $T_n$, the question naturally arises how the path from $\vz$ to $\vo$ with the smallest first-passage time typically behaves. In particular, how long is this path (here length means the number of edges along the path), and how are the ``backsteps'' distributed along it. Let us denote this path by $\Gamma_n$. This question may also be interesting from the point of view of accessibility percolation. Though strictly speaking not part of the mathematical formulation of accessibility percolation, shorter paths are considered more biologically feasible. Hence, an important question for unoriented accessibility percolation on the hypercube is how much longer typical accessible paths are in this case compared to oriented accessibility percolation.

We propose the following way to describe the asymptotic properties of $\Gamma_n$: Run a simple random walk on $\hc{n}$ starting at $\vz$ with rate $n$ for $\ln\left(1+\sqrt{2}\right)$ time, and condition on the event that the walk stops at $\vo$. Let $\sigma_n$ denote the traversed path.
\begin{theorem}\label{thm:sigmamin}
Any asymptotically almost sure property of $\sigma_n$ is also an asymptotically almost sure property of $\Gamma_n$. In particular, the length of $\Gamma_n$ is asymptotically almost surely $\sqrt{2}\ln\left(1+\sqrt{2}\right)n \pm o(n)$.
\end{theorem}
In applying Theorem \ref{thm:sigmamin}, it is helpful to note that each coordinate of a simple random walk on $\hc{n}$ with rate $n$ is an independent simple random walk on $\{0, 1\}$ with rate $1$.

The remainder of the paper will be structured in the following way: In Section 2 we define the BTP and describe our stochastical sandwiching of Richardson's model. At the end of this section, we give an outline of the proof of Theorem \ref{thm:unorientedFPP}. This proof is divided into three steps, which are shown in Sections \ref{sec:proofuncontestedBTP}, \ref{sec:proofBSestimates} and \ref{sec:bootstrap} respectively. Lastly, in Section \ref{sec:proofofsigmamin} we give a short proof of Theorem \ref{thm:sigmamin} based on ideas from the preceding section.

\section{Richardson's model, the BTP, and uncontested particles}

We first give an overview of the technique used by Durrett to obtain the lower bound on $T_n$ in \cite{fp93}. To accommodate Theorem \ref{thm:uncontestedBTP} below, we present this technique in terms of a general graph $G$ rather than just the hypercube. We remark that though Durrett only defined the branching translation process for the hypercube, the process can be extended to a general graph unambiguously. We let $v_0$ denote a fixed vertex in $G$. For simplicity, we will assume that $G$ is finite, connected and simple.

The \emph{branching translation process} (BTP), as introduced by Durrett, is a branching process on $G$ defined in the following way: At time $0$ we place a particle at $v_0$. After this, each existing particle generates offspring independently at rate equal to the degree of the vertex it is placed at. Each offspring is then placed with uniform probability at any neighboring vertex. Equivalently, each existing particle generates offspring at each neighboring vertex independently with rate $1$. For a fixed $G$ and fixed location of the first particle $v_0 \in G$, we let $Z(v, t)$ denote the  number of particles at vertex $v$ at time $t$ in the BTP (originating at $v_0$) and define $m(v, t) = \mathbb{E} Z(v, t)$. One can observe that $\{Z(\cdot, t)\}_{t\geq 0}$ is a Markov process with the initial value $Z(v, 0) = \delta_{v,v_0}$ and where, for each vertex $v$, the transition $\{Z(\cdot) \rightarrow Z(\cdot)+\delta_{\cdot,v}\}$ occurs at rate $\sum_{w\in N(v)} Z(w)$ where $N(v)$ denotes the neighborhood of $v$. It can be noted that in \cite{fp93}, the BTP was formally defined as this Markov process. However, this way to describe the states contains an insufficient amount of information for our applications since there is no way to discern ancestry. We will return to the problem of formally defining the state space of the BTP in Section \ref{sec:proofuncontestedBTP}. For now, the reader not satisfied with the informal definition of the BTP given here is free to consider any state space in which the particles can be individually identified and for each particle except the first, it is possible to determine its parent.

Below, we will use the terms ancestor and descendant of a particle to denote the natural partial order of particles generated by the BTP. For convenience, we use the convention that a particle is both an ancestor and a descendant of itself. We will sometimes write $x \geq y$ to denote that $x$ is a descendant of $y$, and $x \leq y$ to denote that $x$ is an ancestor of $y$. The terms parent and child are defined in the natural way. In order to indicate the location of a child of a particle $x$, we will sometimes use the term $e$-child of $x$ to denote a child of $x$ which at the time of its birth was displaced along an edge $e$. We define the \emph{ancestral line} of a particle $x$ as the ordered set of all ancestors of $x$ (including $x$ itself). If $\sigma$ is the path obtained by following the locations of the vertices along the ancestral line of a particle $x$, then we say that the ancestral line of $x$ \emph{follows $\sigma$}, and we say that the ancestral line of $x$ is \emph{simple} if this path is simple. In certain parts of our proof we will need to consider BTPs where the location of the initial particle can vary. In that case, we will refer to a BTP where the original particle is placed at $v$ as the BTP \emph{originating at $v$}.

As pointed out in \cite{fp93}, the BTP stochastically dominates Richardson's model in the sense that, for a common starting vertex $v_0$, the models can be coupled in such a way that $R(v, t) \leq Z(v, t)$ for all $v\in G$ and $t\geq 0$. This is clear from a comparison of the transition rates of $Z$ and $R$. However, for our applications we need to consider this relation more closely. To this end, we imagine that we partition the particles in the BTP into two sets, which we call the set of \emph{alive} particles and the set of \emph{ghosts}. We stress that the state of a particle is decided at the time of its birth, and is then never changed. The original particle is placed in the set of alive particles. After this, whenever a new particle is born it is placed in the set of ghosts if its parent is a ghost or if its location is already occupied by an alive particle, and placed in the set of alive particles otherwise. Clearly, the subprocess of the BTP consisting of all alive particles initially contains one particle, located at $v_0$, and it is straightforward to see that the rate at which alive particles are born at a given vertex $v$ equals the number of adjacent vertices that contain alive particles if $v$ does not currently contain an alive particle, and $0$ if it does. As this is the same transition rate as for the corresponding transition in Richardson's model, we can consider Richardson's model as the subprocess of the BTP consisting of all alive particles. In a sense, for an observer not able to see the ghosts, the BTP will look like Richardson's model. Hence, with this coupling, the time at which a vertex gets infected is equal to one of the arrival times at the corresponding vertex in the full BTP, though not necessarily the first. We may here note that as at most one particle can be alive at each vertex, we can interpret $R(v, t)$ as the number of alive particles at $v$ at time $t$.

A simplified version of the proof of the lower bound on $T_n$ in \cite{fp93} can now be summarized as follows: Consider a BTP on $\hc{n}$ originating at $\vz$. Since the BTP dominates Richardson's model it suffices to show that with probability $1-o(1)$, no particle occupies $\vo$ at time $\ln\left(1+\sqrt{2}\right)-\varepsilon$ for all $\varepsilon>0$ fixed. This is shown by a first moment method. It follows from standard methods in the theory of continuous-time Markov chains that $m(v, t)$ is the unique solution to the initial value problem
\begin{equation}\label{eq:DEm}
\begin{split}
\frac{d}{dt} m(v, t) &= \sum_{w \in N(v)} m(w, t), \; t>0\\
m(v, 0) &= \delta_{v, v_0}.
\end{split}
\end{equation}
In the case where $G=\hc{n}$ and $v_0=\vz$, it is straightforward to check that the solution to \eqref{eq:DEm} is
\begin{equation}\label{eq:mvthypercube}
m(v, t) = \left(\sinh t\right)^{\abs{v}} \left(\cosh t \right)^{n-\abs{v}}
\end{equation}
and hence $m(\vo, t) = \left(\sinh t\right)^n$. Clearly, this tends to $0$ as $n\rightarrow \infty$ for any $t < \sinh^{-1} 1 = \ln\left(1+\sqrt{2}\right)$, as desired.

Inspired by the coupling between Richardson's model and the BTP as above, we introduce the notion of a particle being uncontested. For a particle $x$ in a BTP, we let $c(x)$ denote the number of pairs of distinct particles $y, z$ such that
\begin{itemize}
\item $y$ is an ancestor of $x$
\item $y$ and $z$ occupy the same vertex
\item $z$ was born before $y$.
\end{itemize}
Note that, according to our definition of ancestor, it is allowed for $y$ to be equal to $x$. We let $a(x)$ denote the number of such pairs where $z$ is an ancestor of $x$, and let $b(x)$ denote the number of pairs where $z$ is not an ancestor of $x$. Clearly $a(x)+b(x)=c(x)$. We say that a particle $x$ is \emph{uncontested} if $c(x)=0$.

\begin{lemma}\label{lemma:basicprop}
We have the following properties:
\begin{enumerate}[label=\roman*)]
\item $a(x)=0$ if and only if the ancestral line of $x$ is simple
\item if a particle is uncontested, then it is the first particle to be born at its location
\item if a particle is uncontested, then it is alive.
\end{enumerate}
\end{lemma}
\begin{proof}
$i)$ This is obvious. $ii)$ If some particle $z$ was born before $x$ at a vertex, then the pair $(x, z)$ is counted in $c(x)$. $iii)$ For any ghost $x$ in the BTP, there must exist an earliest ancestor $y$ which is a ghost. As the original particle is, by definition, alive, $y$ must have a parent in the BTP. As the parent of $y$ is alive but $y$ is a ghost, the vertex occupied by $y$ must already have been occupied by some alive particle $z$ at the time of birth of $y$. The pair $(y, z)$ is then counted in $c(x)$.
\end{proof}

The third property is of particular interest as it allows us to express a lower bound on Richardson's model in terms of the BTP. Letting $Z_k(v, t)$ denote the number of particles $x$ at vertex $v$ at time $t$ such that $c(x)=k$, we conclude that 
\begin{equation}\label{eq:stochdom}
Z_0 \overset{d}{\leq} \text{Richardson's model} \overset{d}{\leq} Z,
\end{equation}
and with the proposed coupling between BTP and Richardson's model above we even have $Z_0 \leq R \leq Z$. However, it should be noted that, unlike $Z$ and $R$, there is no reason why $Z_0(v)$ could not remain $0$ forever. In fact, with the exception of the case where $G$ is a chain of length $1$, this occurs with positive probability. In order to see this, one can observe that if the first particle to arrive at a vertex is contested, which occurs with positive probability, then this particle will prevent all subsequent particles from being uncontested. On the other hand, in the event that $Z_0(v)$ is eventually non-zero, it follows from the second and third properties in Lemma \ref{lemma:basicprop} that the uncontested particle must have been the first particle at $v$ and that this particle must have been alive. Hence, either $Z_0(v)$ remains $0$ forever, or the time of the first arrival at $v$ coincides in all three models.

\subsection{Outline of proof of Theorem \ref{thm:unorientedFPP} }
For each vertex $v$ and $t\geq 0$, we define $A(v, t)$ and $B(v, t)$ as the expected value of $\sum_x a(x)$ and $\sum_x b(x)$ respectively, where the sums goes over all particles at vertex $v$ at time $t$ in the BTP. We similarly define $S(v, t)$ as the expected number of particles at vertex $v$ at time $t$ with simple ancestral lines, that is, the expected number of particles $x$ at $v$ at time $t$ such that $a(x)=0$. The core of finding upper bounds on the first-passage time using the BTP is the following theorem, which will be shown in Section \ref{sec:proofuncontestedBTP}:
\begin{theorem}\label{thm:uncontestedBTP}
Let $G$ be a finite connected simple graph.
Consider the BTP on $G$ originating at $v_0$, and let $Z_0(v, t)$, $B(v, t)$ and $S(v, t)$ be as above. Then, for any vertex $v$ and $t\geq 0$ we have
\begin{equation}\label{eq:uncontestedBTP}
\mathbb{P}\left(Z_0(v, t) > 0\right) \geq S(v, t) e^{- \frac{B(v, t)}{S(v, t)}}.
\end{equation}
\end{theorem}
In essence, Theorem \ref{thm:uncontestedBTP} states that if, at a time $t$, the expected number of particles with simple ancestral line at $v$ in the BTP is bounded away from $0$, and if $B(v, t)$ is bounded, then with probability bounded away from $0$ there is a particle at $v$ at this time such that $a(x)=b(x)=0$. Using the relation between the BTP and Richardson's model in \eqref{eq:stochdom}, this immediately implies a lower bound on the probability that the first-passage time from $v_0$ to $v$ in $G$ is at most $t$. We remark that while the left-hand side of \eqref{eq:uncontestedBTP} certainly is increasing in $t$, the right-hand side is generally not, and instead typically attains a maximum for $t$ such that $m(v, t) \approx 1$.

We now apply this result to the hypercube. We let $G=\hc{n}$, $v_0=\vz$ and $t=\fpptime \defeq \ln\left(1+\sqrt{2}\right)$. In this case, the quantities $A(\vo, \fpptime)$ and $B\left(\vo, \fpptime\right)$ can be expressed analytically in a similar manner to the variance calculations for the BTP in \cite{fp93}. This will be done in Section \ref{sec:proofBSestimates}. The result of this can be summarized as follows: 
\begin{proposition} \label{prop:BSestimates} For $\fpptime=\ln\left(1+\sqrt{2}\right)$, we have
\begin{align}
A\left(\vo, \fpptime\right) &= \frac{ \fpptime}{\sqrt{2}} + o(1) = 0.623\dots + o(1)\\
B\left(\vo, \fpptime\right) &= \fpptime + \frac{1}{3-2\sqrt{2}} + o(1) = 6.709\dots + o(1).
\end{align}
\end{proposition}
In order to bound $S(\vo, \fpptime)$, we observe that $A(v, t)$ is an upper bound on the expected number of particles at $v$ at time $t$ whose ancestral lines are not simple. This follows directly from the definition of $A(v, t)$ as $a(x)$ is an upper bound on the indicator function for the event that $a(x)$ is non-zero. We conclude that
\begin{equation}
m(v, t) - A(v, t) \leq S(v, t) \leq m(v, t),
\end{equation}
and in particular, $1-\frac{\fpptime}{\sqrt{2}} - o(1) = 0.376\dots - o(1) \leq S(\vo, \fpptime) \leq 1$.

Plugging these values into Theorem \ref{thm:uncontestedBTP} we conclude the following:
\begin{corollary}\label{cor:Tnnotzero}
Let $T_n$ denote the first-passage time from $\vz$ to $\vo$ in $\hc{n}$ and let $\fpptime=\ln\left(1+\sqrt{2}\right)$. There exists a constant $\varepsilon>0$ such that $\mathbb{P}(T_n \leq \fpptime)\geq \varepsilon$ for all $n$, and in particular\newline $\liminf_{n\rightarrow\infty}\mathbb{P}(T_n \leq \fpptime) \geq 6.9\cdot 10^{-9}$.
\end{corollary}
\begin{proof}
The asymptotic lower bound on $\mathbb{P}(T_n \leq \fpptime)$ is obtained directly from Theorem \ref{thm:uncontestedBTP} and Proposition \ref{prop:BSestimates}. From this, the uniform bound follows by the observation that $\mathbb{P}(T_n \leq \fpptime)$ is non-zero for all $n$.
\end{proof}

For our applications, we will need a more technical version of Corollary \ref{cor:Tnnotzero}, Proposition \ref{prop:H2prime}, but other than that we are done with the BTP given this result. It may seem like Corollary \ref{cor:Tnnotzero} is far from our claimed result of convergence in $L^p$-norm, but given this result there are in fact a number of different ways to show that $T_n$ converges to $\fpptime$ at least in probability, using the self-similar structure of the hypercube. One could for instance apply the ideas by Bollob\'{a}s and Kohayakawa in \cite{bk97}. In this paper we will instead apply a bootstrapping argument similar to one given in \cite{hm13}, which has the benefit of letting us get good bounds on the $L^p$-norms of $T_n-\fpptime$. This will be shown in Section \ref{sec:bootstrap}, completing the proof of Theorem \ref{thm:unorientedFPP}.

\section{Proof of Theorem \ref{thm:uncontestedBTP}}\label{sec:proofuncontestedBTP}

Before proceeding with the proof, we need to discuss the parametrization of the BTP more carefully. For a BTP originating at a vertex $v$, a particle is identified by a finite sequence $\{e_1 t_1 e_2 t_2\dots e_k t_k\}$ where $e_1, e_2, \dots, e_k$ are edges forming a path that starts at $v$ and $t_1, t_2, \dots, t_k$ are positive real numbers. The original particle is identified by $\{\}$, the empty sequence. For any other particle $x$, $e_1 e_2 \dots e_k$ denotes the edges along the path followed by the ancestral line of $x$, and if $x_0, x_1, \dots, x_k=x$ are the ancestors of $x$ in ascending order, then for each $1\leq i \leq k$, we have $t_i$ equal to the time from the birth of $x_{i-1}$ to the birth of $x_i$. It is easy to see that such a sequence uniquely defines the location and birth time of $x$. In particular, as, almost surely, no two particles are born at exactly the same time, this means that this representation is unique for each particle in the BTP. Note that this means that the parent of $x=\{e_1 t_1 e_2 t_2\dots e_k t_k\}$ is $\{e_1 t_1 e_2 t_2\dots e_{k-1} t_{k-1}\}$. More generally, the ancestors of $x$ are the prefixes of $x$ of even length. By a BTP originating at a vertex $v$ we formally mean a random set of particles, which is interpreted as the set of all particles that will ever be born in the BTP, and, of course, whose distribution is given according to the transition rates as described above. We remark that this means that the event that a particle $x=\{e_1 t_1 e_2 t_2\dots e_k t_k\}$ exists is interpreted as the event that the original particle has a $e_1$-child at time $t_1$, that this child has an $e_2$-child at time $t_1+t_2$ and so on.

Below will use $\oplus$ to denote concatenation of sequences. For instance, if $y$ is a child of $x$, born a time $t$ after its parent and displaced along the edge $e$, then we may write $y = x \oplus \{e t\}$. For a sequence $a$ and a set of sequences $B$, we define $a \oplus B = \left\{a\oplus b\middle\vert b\in B\right\}$.

It is easy to see that, in a BTP, each vertex can at most contain one uncontested particle, see for instance property $ii)$ in Lemma \ref{lemma:basicprop}. This means that the probability that a vertex $v$ contains an uncontested particle at time $t$ is equal to the expected number of such particles. Hence the conclusion of Theorem \ref{thm:uncontestedBTP} basically states that among the particles at $v$ at time $t$ such that $a(x)=0$, the probability that $b(x)=0$ is on average at least $\exp\left(-\frac{B(v, t)}{S(v, t)}\right)$. In principle, it is possible to show this by considering the conditional distributions of $b(x)$ given the event that the particle $x$ exists in the BTP. However, it is not formally possible by the usual definitions of conditional expectation and conditional distribution to condition on the event that a particle exists in the BTP since the event occurs with probability $0$ and the particle itself is not the output of some well-defined random variable. In order to solve this problem, we need some ideas from Palm theory, and, in particular, the following special case of the Slivnyak-Mecke formula. The proof of this can be found in various text books on point processes. See for instance Corollary 3.2.3 in \cite{sw08}.
\begin{theorem}(Slivnyak-Mecke formula)
Let $\mathbf{T}$ be a Poisson point process on the positive part of the real line with with constant intensity $1$. Let $G$ be a function mapping pairs $(T, t)$ where $T$ is a discrete subset of $\mathbb{R}_+$ and $t\in T$ to non-negative real numbers. Then
\begin{equation}\label{eq:slivnyak-mecke}
\mathbb{E} \sum_{t\in \mathbf{T}} G(\mathbf{T}, t) = \int_0^\infty \mathbb{E} G\left( \mathbf{T} \cup \{t\}, t\right)\,dt.
\end{equation}
\end{theorem}

If instead of a Poisson process on $\mathbb{R}_+$, we imagine $\mathbf{T}$ being a random subset of a finite, or even countable set, then we clearly have
\begin{equation}
\mathbb{E} \sum_{t\in \mathbf{T}} G(\mathbf{T}, t) = \sum_t \mathbb{P}\left(t \in \mathbf{T}\right) \mathbb{E}\left[ G(\mathbf{T}, t) \middle\vert t\in \mathbf{T}\right].
\end{equation}
By the standard way to translate this statement, if $\mathbf{T}$ is a Poisson process on $\mathbb{R}_+$ with constant intensity $1$, then we would expect the sum over $t$ to translate to an integral and $\mathbb{P}\left(t \in \mathbf{T}\right)$ to $dt$, the Lebesgue measure on $\mathbb{R}_+$. Hence, the theorem states that if $\mathbf{T}$ is a Poisson process as above, then we should translate $\mathbb{E}\left[ G(\mathbf{T}, t) \middle\vert t\in \mathbf{T}\right]$ to $\mathbb{E} G\left( \mathbf{T} \cup \{t\}, t\right)$, and so we may interpret $\mathbf{T} \cup \{t\}$ as the conditional distribution of $\mathbf{T}$ given $t\in \mathbf{T}$.

The following lemma proves a corresponding result for the BTP. In a similar manner as above, we may interpret the lemma as that, conditioned on the event that a particle $x^{z_1, \dots, z_l}$ exists in the BTP $\mathbf{X}_0$, the conditional distribution of the process is given by $\mathbf{X}^{z_1, \dots, z_l}$, where $x^{z_1, \dots, z_l}$ and $\mathbf{X}^{z_1, \dots, z_l}$ are as defined below. This result may be well-known from the properties of more general processes.
\begin{lemma}\label{lemma:palmBTP}
Let $\sigma$ be a path of length $l\geq 1$. We denote the vertices along the path $v_0, \dots, v_l$ and the edges $\sigma_1, \dots, \sigma_l$. Let $\mathbf{X}_0, \mathbf{X}_1, \dots, \mathbf{X}_l$ be independent branching translation processes where $\mathbf{X}_i$ for $0\leq i \leq l$ is a BTP originating at vertex $v_i$.  Let $f$ be a function taking pairs $(X, x)$, $X$ a realization of a BTP and $x$ a particle in $X$, to non-negative real numbers. Let $V_{\sigma} = V_{\sigma}\left( X \right)$ denote the set of particles at vertex $v_l$ (no matter when they are born) whose ancestral line follows $\sigma$. Then, we have
\begin{equation}
\mathbb{E} \sum_{x \in V_{\sigma} (\mathbf{X}_0)} f(\mathbf{X}_0, x) = \int_0^\infty \dots \int_0^\infty \mathbb{E}\,f(\mathbf{X}^{z_1, \dots, z_l}, x^{z_1, \dots, z_l})\,dz_1\dots\,dz_l
\end{equation}
where
\begin{equation}
\mathbf{X}^{z_1, \dots, z_l} = \mathbf{X}_0 \cup \left(\{\sigma_1 z_1\}\oplus\mathbf{X}_1\right) \cup \left(\{\sigma_1 z_1 \sigma_2 z_2\}\oplus\mathbf{X}_2\right) \cup \dots \cup \left(\{\sigma_1 z_1 \sigma_2 z_2 \dots \sigma_l z_l\}\oplus\mathbf{X}_l\right)
\end{equation}
and $x^{z_1, \dots, z_l}=\{\sigma_1 z_1 \sigma_2 z_2\dots \sigma_l z_l\}$.
\end{lemma}

\begin{proof}
For a vertex $v$ and an edge $e$ we let $e\ni v$ denote that $v$ is one of the end points of $e$. For each edge $e\ni v_0$, we let $\mathbf{T}_e$ denote the set of birth times of the $e$-children of the original particle in $\mathbf{X}_0$. Clearly, $\mathbf{T}_e$ for $e\ni v_0$ are independent Poisson processes on $\mathbb{R}_+$ with constant intensity $1$.

A central property of the BTP is that, after a particle is born, the set of its descendants is itself distributed as a BTP. Furthermore, this subprocess is then independent of the behavior of any other particle. Hence we can express $\mathbf{X}_0$ recursively by
\begin{equation}\label{eq:Xorec}
\mathbf{X}_0 = \bigcup_{e \ni v_0} \bigcup_{t_i\in \mathbf{T}_e} \{e t_i\}\oplus \mathbf{Y}_{e,i}
\end{equation}
where for each $e\ni v_0$ and each $i=1, 2, \dots$, we have $\mathbf{Y}_{e,i}$ independently distributed as a BTP originating at the vertex opposite to $v_0$ along $e$. For any discrete set $T \subset \mathbb{R}_+$ we let $\mathbf{X}_0(T)$ denote the random variable obtained by replacing $\mathbf{T}_{\sigma_1}$ by $T$ in \eqref{eq:Xorec}. Then $\mathbf{X}_0(\cdot)$ is a random function independent of $\mathbf{T}_{\sigma_1}$, and we have $\mathbf{X_0}=\mathbf{X_0}(\mathbf{T}_{\sigma_1})$. Note that, by independence, $\mathbf{X}_0(T)$ is a version of the conditional distribution of $\mathbf{X}_0$ given $\mathbf{T}_{\sigma_1}=T$.

For each $T$ as above and $t\in T$, we define
\begin{equation}
F(T) = \mathbb{E}\sum_{x \in V_{\sigma} (\mathbf{X}_0(T))} f(\mathbf{X}_0(T), x)
\end{equation}
\begin{equation}
F(T, t) = \mathbb{E}\sum_{\substack{x \in V_{\sigma} (\mathbf{X}_0(T))\\x \geq \{\sigma_1 t\}}} f(\mathbf{X}_0(T), x).
\end{equation}
It is clear from the definition that, for any fixed $T$, we have $F(T) = \sum_{t\in T} F(T, t)$. Furthermore, as $\mathbf{T}_{\sigma_1}$ and $\mathbf{X}_0(\cdot)$ are independent we have $\mathbb{E} F(\mathbf{T}_{\sigma_1}) = \mathbb{E} \sum_{x \in V_{\sigma} (\mathbf{X}_0)} f(\mathbf{X}_0, x)$. Hence by the Slivnyak-Mecke formula we have
\begin{equation}
\mathbb{E} \sum_{x \in V_{\sigma} (\mathbf{X}_0)} f(\mathbf{X}_0, x) = \mathbb{E} \sum_{t\in \mathbf{T}_{\sigma_1}} F(\mathbf{T}_{\sigma_1}, t)=\int_0^\infty \mathbb{E} F(\mathbf{T}_{\sigma_1}\cup \{z_1\}, z_1) \, dz_1.
\end{equation}
By independence of $\mathbf{X}_0(\cdot)$ and $\mathbf{T}_{\sigma_1}\cup \{z_1\}$ we can conclude that
\begin{equation}\label{eq:X0Tcupz1}
\mathbb{E} \sum_{x \in V_{\sigma} (\mathbf{X}_0)} f(\mathbf{X}_0, x) =
\int_0^\infty \mathbb{E}\sum_{\substack{x \in V_{\sigma} (\mathbf{X}_0(\mathbf{T}_{\sigma_1}\cup\{z_1\} ))\\x \geq \{\sigma_1 z_1\}}} f(\mathbf{X}_0(\mathbf{T}_{\sigma_1}\cup\{z_1\}), x) \,dz_1.
\end{equation}

Let us now consider the random process $\mathbf{X}_0(\mathbf{T}_{\sigma_1}\cup\{z_1\})$. We can interpret the expression for $\mathbf{X}_0$ in \eqref{eq:Xorec} and the subsequent definition of $\mathbf{X}_0(T)$ as that these processes are generated by first determining the birth time for each child of the original particle, and then for each child independently generating a BTP which determines its descendants. When seen in this light, it is clear that the only difference between $\mathbf{X}_0$ and $\mathbf{X}_0(\mathbf{T}_{\sigma_1}\cup\{z_1\})$ is that the latter has an additional particle in generation $1$. Hence, $\mathbf{X}_0(\mathbf{T}_{\sigma_1}\cup\{z_1\})$ has the same distribution as $\mathbf{X}_0\cup \left( \{\sigma_1 z_1\} \oplus \mathbf{X}_1\}\right)$, and so we can replace $\mathbf{X}_0(\mathbf{T}_{\sigma_1}\cup\{z_1\})$ in \eqref{eq:X0Tcupz1} by this other random process.

Letting $\tilde{\sigma} = \{\sigma_2, \sigma_3, \dots, \sigma_l\}$, we note that the subset of elements in $V_\sigma \left(\mathbf{X}_0\cup \left( \{\sigma_1 z_1\} \oplus \mathbf{X}_1\}\right)\right)$ that are descendants of $\{\sigma_1 z_1\}$ is precisely the set $ \{ \sigma_1 z_1\} \oplus V_{\tilde{\sigma}}\left(\mathbf{X}_1\right)$. Hence \eqref{eq:X0Tcupz1} simplifies to
\begin{equation}\label{eq:recformula}
\mathbb{E} \sum_{x \in V_{\sigma} (\mathbf{X}_0)} f(\mathbf{X}_0, x) = \int_0^\infty \mathbb{E}\sum_{x \in V_{\tilde{\sigma}}(\mathbf{X}_1) } f\left(\mathbf{X}_0\cup \left( \{\sigma_1 z_1\} \oplus \mathbf{X}_1\}\right), \{\sigma_1 z_1\}\oplus x\right) \,dz_1.
\end{equation}

The lemma follows by induction. If $l=1$, then the only particle in $V_{\tilde{\sigma}}(\mathbf{X}_1)$ is $\{\}$, the original particle in $\mathbf{X}_1$, and so equation \eqref{eq:recformula} simplifies to
\begin{equation}
\mathbb{E} \sum_{x \in V_{\sigma} (\mathbf{X}_0)} f(\mathbf{X}_0, x) = \int_0^\infty \mathbb{E} f(\mathbf{X}^{z_1}, \{\sigma_1 z_1\}) \,dz_1
\end{equation}
as desired.

Now assume $l > 1$. By the induction hypothesis we have for any non-negative function $\tilde{f}$
\begin{equation}\label{eq:indhyp}
\mathbb{E} \sum_{x \in V_{\tilde{\sigma}}(\mathbf{X}_1)} \tilde{f}(\mathbf{X}_1, x) = \int_0^\infty \dots \int_0^\infty \mathbb{E}\tilde{f}( \tilde{\mathbf{X}}^{z_2,\dots,z_l}, \tilde{x}^{z_2,\dots,z_l})\,dz_2\dots\,dz_l,
\end{equation}
where 
\begin{equation}
\tilde{\mathbf{X}}^{z_2, \dots, z_l} = \mathbf{X}_1 \cup \left(\{\sigma_2 z_2\}\oplus\mathbf{X}_2\right) \cup \left(\{\sigma_2 z_2 \sigma_3 z_3\}\oplus\mathbf{X}_3\right) \cup \dots \cup \left(\{\sigma_2 z_2 \sigma_3 z_3 \dots \sigma_l z_l\}\oplus\mathbf{X}_l\right)
\end{equation}
and $\tilde{x}^{z_2,\dots,z_l} = \{\sigma_2 z_2 \sigma_3 z_3 \dots \, \sigma_l z_l\}$.

Let us consider the expression $\mathbb{E} \sum_{x \in V_{\tilde{\sigma}}(\mathbf{X}_1) } f\left(\mathbf{X}_0\cup \left( \{\sigma_1 z_1\} \oplus \mathbf{X}_1\}\right), \{\sigma_1 z_1\}\oplus x\right)$, the integrand on the right-hand side of equation \eqref{eq:recformula}. If we fix $z_1 > 0$ and condition on $\mathbf{X}_0 = X_0$, then $f\left(\mathbf{X}_0\cup \left( \{\sigma_1 z_1\} \oplus \mathbf{X}_1\}\right), \{\sigma_1 z_1\}\oplus x\right)$ is a function of $\mathbf{X}_1$ and $x$ only. By the induction hypothesis,
\begin{align*}
&\mathbb{E}\sum_{x \in V_{\tilde{\sigma}}(\mathbf{X}_1) } f\left(X_0\cup \left( \{\sigma_1 z_1\} \oplus \mathbf{X}_1\}\right), \{\sigma_1 z_1\}\oplus x\right)\\
&\qquad= \int_0^\infty \dots \int_0^\infty \mathbb{E} f( X_0 \cup \{\sigma_1 z_1\}\oplus\tilde{\mathbf{X}}^{z_2,\dots,z_l}, \{\sigma_1 z_1\} \oplus \tilde{x}^{z_2,\dots,z_l})\,dz_2\dots\,dz_l.
\end{align*}
Hence, by integrating this expression over $z_1$ and $\mathbf{X}_0$ we conclude that
\begin{align*}
&\mathbb{E} \sum_{x \in V_{\sigma} (\mathbf{X}_0)} f(\mathbf{X}_0, x)\\
&\qquad= \int_0^\infty \dots \int_0^\infty \mathbb{E} f( \mathbf{X}_0 \cup \{\sigma_1 z_1\}\oplus\tilde{\mathbf{X}}^{z_2,\dots,z_l}, \{\sigma_1 z_1\} \oplus \tilde{x}^{z_2,\dots,z_l})\,dz_1\dots\,dz_l,
\end{align*}
where clearly $\mathbf{X}^{z_1,\dots,z_l}=\mathbf{X}_0 \cup \{\sigma_1 z_1\} \tilde{\mathbf{X}}^{z_2,\dots,z_l}$ and $x^{z_1,\dots,z_l} = \{\sigma_1 z_1\} \oplus \tilde{x}^{z_2,\dots,z_l}$.
\end{proof}

\begin{lemma}\label{lemma:superpoisson}
Let $\mathbf{X}$ be a BTP originating at a vertex $v$. Let $\varphi$ be an indicator function defined over the set of potential particles. If $\varphi(\{\})=0$, then
\begin{equation}
\mathbb{P}\left( \varphi(x)=0\, \forall x\in \mathbf{X}\right) \geq \exp\left(-\mathbb{E}\sum_{x \in \mathbf{X}} \varphi(x)\right).
\end{equation}
\end{lemma}
\begin{proof}
For any particle $x\in \mathbf{X}$, let $\psi(x)$ be the indicator function for the event that $\varphi(y) = 1$ for at least one descendant $y$ of $x$. Clearly, we have $\sum_{x \text{ in gen } 1} \psi(x) \leq \sum_{x\in \mathbf{X}} \varphi(x).$ Furthermore, $\sum_{x \text{ in gen } 1} \psi(x)=0$ if and only if $\sum_{x\in \mathbf{X}} \varphi(x)=0$.

Let $d$ denote the degree of the vertex $v$. Then the particles in generation one are born according to a Poisson process on $\mathbb{R}_+^d$. Conditioned on the particles in generation one, the random variables $\psi(x)$ for each such particle $x$ are independent, and are one with probability only depending on the location and birth time of $x$. Hence, by the random selection property of a Poisson process, the particles in generation one that satisfy $\psi(x)=1$ are also born according to a Poisson process, and, in particular, the number of such particles is Poisson distributed.

We conclude that the probability that $\varphi(x)=0$ for all $x\in \mathbf{X}$ is $e^{ -\mathbb{E} \sum_{x \text{ in gen } 1} \psi(x)}$, which is at least $e^{ -\mathbb{E} \sum_{x \in \mathbf{X}} \varphi(x)}$.
\end{proof}

\begin{proof}[Proof of Theorem \ref{thm:uncontestedBTP}]
For any path $\sigma$ from $v_0$ to $v$, let $S_\sigma(v, t)$ and $B_\sigma(v, t)$ denote the contributions to $S(v, t)$ and $B(v, t)$ respectively from particles whose ancestral lines follow $\sigma$. Similarly, we define $P(v, t)=\mathbb{E} Z_0(v, t)$ and $P_\sigma(v, t)$ the contribution to $P(v, t)$ from particles whose ancestral lines follow $\sigma$. As no two particles at the same vertex can both be uncontested, $Z_0(v, t)$ can only assume the values $0$ and $1$, so $P(v, t)$ is indeed the probability that $Z_0(v, t)$ is non-zero.

We start by considering the case where $\sigma$ is a non-simple path. As $S(v, t)$ is the expected number of particles at vertex $v$ at time $t$ whose ancestral line follows a simple path, it is clear that the contribution to $S(v, t)$ from any non-simple path is zero. Similarly, if the ancestral line of a particle follows a non-simple path, then the particle cannot be uncontested. Hence for any non-simple path $\sigma$ we have $S_\sigma(v, t) = P_\sigma(v, t) = 0$, and trivially $B_\sigma(v, t) \geq 0$.

Let us now fix $\sigma$, a simple path from $v_0$ to $v$. We denote the length of $\sigma$ by $l$. For any realization $X$ of $\mathbf{X}_0$ and $x\in X$, let $T(X, x)$ denote the birth time of $x$. Then it follows from  Lemma \ref{lemma:palmBTP} that
\begin{equation}\label{eq:Ssigma}
S_\sigma(v, t) = \mathbb{E} \sum_{x \in V_{\sigma}(\mathbf{X}_0)} \mathbbm{1}_{T(\mathbf{X}_0, x) \leq t} = \int_0^\infty \dots \int_0^\infty \mathbbm{1}_{z_1+\dots+z_l\leq t}\,dz_1\dots\,dz_l.
\end{equation}
In order to express $B_\sigma$ and $P_\sigma$ in a similar manner, we need to revise our notation. Strictly speaking, $b(x)$ is a function not only of a particle, but also of the realization of the BTP. Following the convention we have used earlier in this section, we will now denote this quantity by $b(\mathbf{X}_0, x)$. Using this notation we have
\begin{align}
B_\sigma(v, t) &= \mathbb{E}\sum_{x\in V_{\sigma}(\mathbf{X}_0)} \mathbbm{1}_{T(\mathbf{X}_0, x) \leq t} b(\mathbf{X}_0, x)\\
P_\sigma(v, t) &= \mathbb{E}\sum_{x\in V_{\sigma}(\mathbf{X}_0)} \mathbbm{1}_{T(\mathbf{X}_0, x) \leq t} \mathbbm{1}_{b(\mathbf{X}_0, x)=0}.
\end{align}
Hence, again by Lemma \ref{lemma:palmBTP}
\begin{align}
B_\sigma(v, t) &= \int_0^\infty \dots \int_0^\infty \mathbbm{1}_{z_1+\dots+z_l\leq t} \mathbb{E}\left[ b(\mathbf{X}^{z_1, \dots,z_l}, x^{z_1,\dots,z_l})\right]  \,dz_1\dots\,dz_l\label{eq:Bsigma}\\
\intertext{and}
P_\sigma(v, t) &= \int_0^\infty \dots \int_0^\infty \mathbbm{1}_{z_1+\dots+z_l\leq t} \mathbb{P}\left( b(\mathbf{X}^{z_1, \dots, z_l}, x^{z_1,\dots,z_l})=0\right)  \,dz_1\dots\,dz_l\label{eq:Psigma}.
\end{align}

Fix $z_1, \dots, z_l > 0$ such that $z_1+\dots+z_l \leq t$ and consider the random variable $b(\mathbf{X}^{z_1, \dots, z_l}, x^{z_1,\dots,z_l})$. As $\sigma$ is a simple path, it follows that $b(\mathbf{X}^{z_1, \dots, z_l}, x^{z_1,\dots,z_l})$ is equal to the number of particles $x \in \mathbf{X}^{z_1,\dots,z_l}$ such that, for some $0 \leq i \leq l$, $x$ is born at the vertex $v_i$ before time $\sum_{k=1}^i z_k$. This means that for appropriate indicator functions $\varphi_0, \dots, \varphi_l$ we have
\begin{equation}
b(\mathbf{X}^{z_1, \dots, z_l}, x^{z_1,\dots,z_l}) = \sum_{i=0}^l \sum_{x \in \mathbf{X}_i} \varphi_i(x).
\end{equation}
As the original particles in $\mathbf{X}_0, \dots, \mathbf{X}_l$ correspond to ancestors of $x^{z_1,\dots,z_l}$, these are never counted in $b$ and hence the corresponding indicator functions are always zero. Furthermore, as $\mathbf{X}_0, \dots, \mathbf{X}_l$ are independent processes, we have by Lemma \ref{lemma:superpoisson}
\begin{align}
\begin{split}
\mathbb{P}\left( b(\mathbf{X}^{z_1, \dots, z_l}, x^{z_1,\dots,z_l})=0 \right) &= \prod_{i=0}^l \mathbb{P}\left( \varphi_i(x)=0 \, \forall x\in \mathbf{X}_i\right)\\
&\geq \prod_{i=0}^l \exp\left( -\mathbb{E} \sum_{x\in \mathbf{X}_i} \varphi_i(x)\right)\\
&= \exp\left( -\mathbb{E} b(\mathbf{X}^{z_1, \dots, z_l}, x^{z_1,\dots,z_l}) \right)
\end{split}
\end{align}
Hence, by \eqref{eq:Psigma},
\begin{equation}
P_\sigma(v, t) \geq \int_0^\infty \dots \int_0^\infty \mathbbm{1}_{z_1+\dots+z_l\leq t} \exp\left(- \mathbb{E}b(\mathbf{X}^{z_1, \dots, z_l}, x^{z_1,\dots,z_l})\right)  \,dz_1\dots\,dz_l.\label{eq:Psigma2}
\end{equation}

Let $r_0\in \mathbb{R}$ be fixed. By convexity we have $e^{-r} \geq e^{-r_0}(1+r_0)-e^{-r_0}r$. Applying this inequality to the integrand in \eqref{eq:Psigma2} and comparing to \eqref{eq:Ssigma} and \eqref{eq:Bsigma} we get, for any simple path $\sigma$,
\begin{equation}\label{eq:pathwiseineq}
P_\sigma(v, t) \geq e^{-r_0}(1+r_0)S_\sigma(v, t)-e^{-r_0} B_\sigma(v, t).
\end{equation}
As remarked, for non-simple paths $\sigma$ we have $P_\sigma=S_\sigma=0$ and $B_\sigma \geq 0$, so clearly \eqref{eq:pathwiseineq} holds for all paths $\sigma$ from $v_0$ to $v$. Summing this inequality over all such paths $\sigma$, we get
\begin{equation}
P(v, t) \geq e^{-r_0}(1+r_0) S(v, t) - e^{-r_0} B(v, t).
\end{equation}
It is easy to verify that the right-hand side is maximized by $r_0= \frac{B(v, t)}{S(v, t)}$, which yields the inequality $P(v, t) \geq S(v, t) e^{-\frac{B(v, t)}{S(v, t)}}$ as desired.
\end{proof}

\section{Proof of Proposition \ref{prop:BSestimates}}\label{sec:proofBSestimates}

Throughout this section we assume that the underlying graph in the BTP is $\hc{n}$, and, unless stated otherwise, the BTP is assumed to originate at $\vz$. We will accordingly let $m(v, t)$ denote the expected number of particles at $v$ at time $t$ for a BTP originating at $\vz$, as given by \eqref{eq:mvthypercube}. In order to simplify notation, we will interpret the vertices of $\hc{n}$ as the elements of the additive group $\mathbb{Z}_2^n$, the $n$-fold group product of $\mathbb{Z}_2$, and we let $e_1, e_2, \dots, e_n\in \mathbb{Z}_2^n$ denote the standard basis. We note that for any fixed vertex $w\in \hc{n}$, the map $v \mapsto v-w$ is a graph isomorphism taking $w$ to $\vz$. Hence, for a BTP originating at $w$, the expected number of particles at $v$ at time $t$ is given by $m(v-w, t)$. While addition and subtraction are equivalent in $\mathbb{Z}_2^n$, we will sometimes make a formal distinction between them in order to indicate direction.

\begin{lemma}\label{lemma:derivatives} For any $t > 0$ and $v\in \mathbb{Q}_n$ we have
\begin{equation}
\frac{d^2}{dt^2} m(v, t) = \sum_{i=1}^n\sum_{j=1}^n m(v+e_i+e_j, t)
\end{equation}
and
\begin{equation}
\frac{1}{2} \frac{d^2}{dt^2} m(v, t)^2 = \sum_{i=1}^n \sum_{j=1}^n m(v+e_i+e_j, t) m(v,t) + m(v+e_i, t)m(v+e_j, t).
\end{equation}
\end{lemma}
\begin{proof}
Recall that $m(v, t)$ satisfies
\begin{equation}
\frac{d}{dt} m(v, t) = \sum_{i=1}^n m(v+e_i, t).
\end{equation}
This directly implies that
\begin{align*}
\frac{d^2}{dt^2} m(v, t) &= \frac{d}{dt} \sum_{i=1}^n m(v+e_i, t)\\
&=\sum_{i=1}^n \frac{d}{dt} m(v+e_i, t)\\
&=\sum_{i=1}^n \sum_{j=1}^n m(v+e_i+e_j, t).
\end{align*}
The second equation now follows from  $\frac{1}{2} \frac{d^2}{dt^2} m(v, t)^2 = m''(v, t) m(v, t) + m'(v, t)m'(v, t)$.
\end{proof}

\begin{lemma}\label{lemma:conv}
Let $s, t\geq 0$ and $v\in \mathbb{Q}_n$. Then
\begin{equation}
\sum_{w\in \mathbb{Q}_n} m(w, s)m(v+w, t) = m(v, s+t).
\end{equation}
\end{lemma}
\begin{proof}
If we condition on the state of the BTP at time $s$, then, at subsequent times, the process can be described as a superposition of independent branching processes, originating from each particle alive at time $s$. For each such process originating from a particle at vertex $w$, we have, by symmetry of $\hc{n}$, that the expected number of particles at vertex $v$ at time $t+s$ is $m(v+w, t)$. Hence
\begin{equation}
\mathbb{E}\left[ Z(v, s+t) \middle\vert Z(s) \right] = \sum_{w \in \hc{n}} Z(w, s) m(v+w, t).
\end{equation}
The lemma follows by taking the expected value of this expression.
\end{proof}

We now turn to the problem of expressing $A(\vo, u)$ and $B(\vo, u)$ in terms of $m(v, t)$. Fix $u>0$ and let $\mathbf{X}$ be a BTP on $\hc{n}$ originating at $\vz$. Let $\mathcal{T}$ denote the random set of triples of particles $(x, y, z)$ in $\mathbf{X}$ such that
\begin{itemize}
\item $x$ is located at $\vo$ at time $u$
\item $y$ is an ancestor of $x$
\item $y$ and $z$ occupy the same vertex
\item $z$ was born before $y$.
\end{itemize}
We furthermore partition this set into $\mathcal{T}_a$, the set of all such triples where $y$ is a descendant of $z$, and $\mathcal{T}_b$, the set of all such triples where $y$ is not a descendant of $z$. For any $x$ at $\vo$ at time $u$ in $\mathbf{X}$, it is clear that $c(x)$ gives the number of triples in $\mathcal{T}$ where the first element is $x$. Hence by summing $c(x)$ over all particles at $\vo$ at time $u$ we obtain the size of $\mathcal{T}$. Similarly we see that $\sum_x a(x)$ and $\sum_x b(x)$ where $x$ goes over all particles $x$ at $\vo$ at time $u$ gives the size of $\mathcal{T}_a$ and $\mathcal{T}_b$ respectively. Hence $A(\vo, u)=\mathbb{E}\abs{\mathcal{T}_a}$, $B(\vo, u)=\mathbb{E}\abs{\mathcal{T}_b}$ and $A(\vo, u) + B(\vo, u) = \mathbb{E}\abs{\mathcal{T}}$.

In the following proposition, we derive explicit expressions for $A(\vo, u)$ and $B(\vo, u)$ by counting the expected number of elements in $\mathcal{T}_a$ and $\mathcal{T}$ respectively. Our argument is reminiscent of the second moment calculation for $Z(\vo, u)$ by Durrett in \cite{fp93}.
\begin{proposition}\label{prop:ABformulas}
For any $u>0$ we have
\begin{equation}\label{eq:Avousumsintegralsgalore}
A(\vo, u) = \sum_{v\in \mathbb{Q}_n} \sum_{i=1}^n \sum_{j=1}^n \int_0^\infty \int_0^\infty \mathbbm{1}_{s+t\leq u} m(v, s) m(\vo-v, u-s-t) m(e_j+e_i, t)\,ds\,dt
\end{equation}
\begin{align}
\begin{split}\label{eq:Bvousumsintegralsgalore}
&A(\vo, u)+B(\vo, u) = \sum_{v\in \mathbb{Q}_n} \sum_{w\in \mathbb{Q}_n} \sum_{i=1}^n \sum_{j=1}^n \int_0^\infty \int_0^\infty \mathbbm{1}_{s+t\leq u} m(v, s) m(\vo-w, u-s-t) \cdot\\
&\qquad \cdot \Big( m(w-v, t) m(w-v-e_i+e_j, t) + m(w-v-e_i, t)m(w-v+e_j, t)\Big)\,ds\,dt.
\end{split}
\end{align}
\end{proposition}
\begin{proof}
Let us start by considering $A(\vo, u)$. For any $(x, y, z)\in \mathcal{T}_a$ there are well-defined particles $c$, the particle subsequent to $z$ in the ancestral line of $x$, and $p$, the parent of $y$. We note that $y$ is not a child of $z$ as then $y$ and $z$ would not be located at the same vertex, hence $c$ must be an ancestor of $p$. This means that the for each triple $(x, y, z)$, the particles $(x, y, z, c, p)$ must be related as illustrated in Graph 1 of Figure \ref{fig:heredity}.

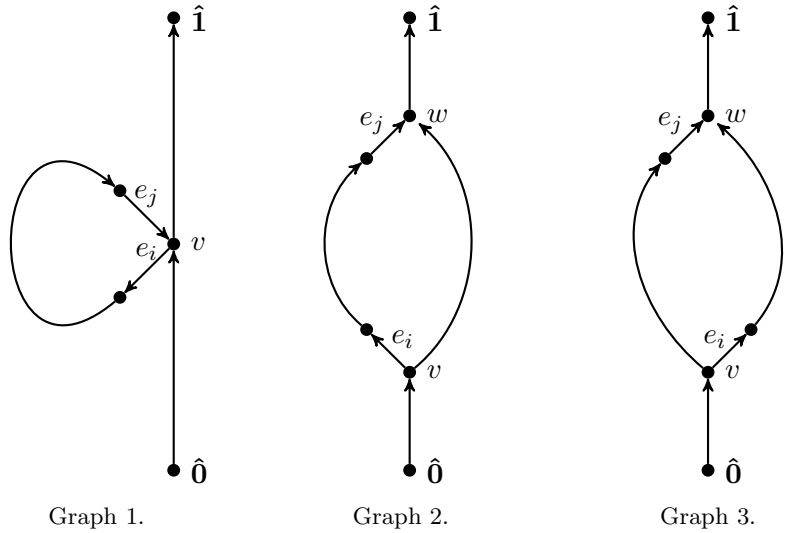
\begin{figure}[ht]
\captionsetup[subfigure]{labelformat=empty}

\begin{center}
%\hspace{1cm}
\subfloat[][Graph 1.]{
\begin{tikzpicture}[->,>=stealth',shorten >=0pt,auto,node distance=3cm,thick,main node/.style={circle,fill=black,draw,minimum size=4pt,inner sep=0pt}]

\node[main node] (vz) [label=right:$\vz$] {};
\node[main node] (v) [above of=vz] [label=right:$v$] {};
\node[main node] (vo) [above of=v] [label=right:$\vo$] {};
\node[main node] (c) [below left of=v,node distance=1.0cm] {};
\node[main node] (p) [above left of=v,node distance=1.0cm] {};

\path
(vz) edge (v)
(v) edge node[above] {$e_i$} (c)
(c) edge [out=220,in=140,looseness=4] (p)
(p) edge node[above] {$e_j$} (v)
(v) edge (vo);
\end{tikzpicture}
}
\hspace{0.85cm}
\subfloat[][Graph 2.]{
\begin{tikzpicture}[->,>=stealth',shorten >=0pt,auto,node distance=2cm,thick,main node/.style={circle,fill=black,draw,minimum size=4pt,inner sep=0pt}]

\node[main node] (vz) [label=right:$\vz$] {};
\node[main node] (v) [above of=vz,node distance=1.3cm] [label=right:$v$] {};
\node[main node] (w) [above of=v,node distance=3.4cm] [label=right:$w$] {};
\node[main node] (vo) [above of=w,node distance=1.3cm] [label=right:$\vo$] {};
\node[main node] (c) [above left of=v,node distance=0.8cm] {};
\node[main node] (p) [below left of=w,node distance=0.8cm] {};

\path
(vz) edge (v)
(v) edge node[above right=-0.1cm] {$e_i$} (c)
(c) edge [out=140,in=220,looseness=1] (p)
(v) edge [out=40,in=-40,looseness=1,shorten >=2pt] (w)
(p) edge node[above left=-0.1cm] {$e_j$} (w)
(w) edge (vo);
\end{tikzpicture}
}
\hspace{1.15cm}
\subfloat[][Graph 3.]{
\begin{tikzpicture}[->,>=stealth',shorten >=0pt,auto,node distance=2cm,thick,main node/.style={circle,fill=black,draw,minimum size=4pt,inner sep=0pt}]

\node[main node] (vz) [label=right:$\vz$] {};
\node[main node] (v) [above of=vz,node distance=1.3cm] [label=right:$v$] {};
\node[main node] (w) [above of=v,node distance=3.4cm] [label=right:$w$] {};
\node[main node] (vo) [above of=w,node distance=1.3cm] [label=right:$\vo$] {};
\node[main node] (c) [above right of=v,node distance=0.8cm] {};
\node[main node] (p) [below left of=w,node distance=0.8cm] {};

\path
(vz) edge (v)
(v) edge node[above left=-0.1cm] {$e_i$} (c)
(v) edge [out=140,in=230,looseness=1] (p)
(c) edge [out=50,in=-40,looseness=1,shorten >=2pt] (w)
(p) edge node[above left=-0.1cm] {$e_j$} (w)
(w) edge (vo);
\end{tikzpicture}
}
%\captionsetup{width=\textwidth}
\caption{Illustration of the possible configurations of ancestral lines of elements in $\mathcal{T}_a$ and $\mathcal{T}$ respectively. Graph 1 shows the configuration of elements in $\mathcal{T}_a$. Here $z$ is located at $v$, $c$ is a child of $z$ at $v+e_i$, $p$ is a descendant of $c$ at $v-e_j$, $y$ is a child of $p$ at $v$, and $x$ a descendant of $y$ at $\vo$. The possible configurations corresponding to elements in $\mathcal{T}$ are shown in Graphs 2 and 3. After the ancestral lines of $x$ and $z$ split, the unique ancestors of $x$ and $z$ are given by the left-most and right-most paths respectively. In both configurations, the last common ancestor of $x$ and $z$, $l$, is located at $v$, the first particle which is an ancestor of precisely one of $x$ and $z$, $c$, is located at $v+e_i$, the parent of $y$, $p$, is located at $w-e_j$, $y$ and $z$ are located at $w$, and $x$ is located at $\vo$.
}\label{fig:heredity}
\end{center}
\end{figure}

Fix $v \in \hc{n}$, $1\leq i, j \leq n$, and infinitesimal time intervals $(s, s+ds]$ and $(s+t, s+t+dt]$ where $0\leq s < s+t < u$. We now count the expected number of such quintuples of particles where the common location of $y$ and $z$ is $v$, the location of $c$ is $v+e_i$, the location of $p$ is $v-e_j$, $c$ is born during $(s, s+ds]$ and $y$ is born during $(s+t, s+t+dt]$. A particle is a potential $z$ if it is located at $v$ at time $s$. For each potential $z$, a potential $c$ is a child of $z$ born at $v+e_i$ during the time interval $(s, s+ds]$. For each pair of a potential $z$ and $c$, a particle is a potential $p$ if it is a descendant of $c$ located at $v-e_j$ at time $s+t$. For each potential triple $(z, c, p)$, a particle is a potential $y$ if it is a child of $p$ born at $v$ during $(s+t, s+t+dt]$. Lastly, for each potential quadruple $(z, c, p, y)$ each particle $x$ at $\vo$ at time $u$ which is a descendant of $y$ forms a triple in $\mathcal{T}_a$. By computing the expected number of potential particles in each step, we see that the expected number of elements in $\mathcal{T}_a$ corresponding to fixed $v, i, j$ and fixed time intervals $(s, s+ds]$ and $(s+t, s+t+dt]$ is
\begin{equation}
m(v, s)\,ds\,m(-e_j-e_i, t)\,dt\,m(v, u-s-t).
\end{equation}
Equation \eqref{eq:Avousumsintegralsgalore} follows by integrating over all $s, t > 0$ such that $s+t<u$ and summing over all $v\in \hc{n}$ and all $1\leq i,j\leq n$.

We now turn to the formula for $\mathbb{E}\abs{\mathcal{T}}$. For each triple $(x, y, z)\in \mathcal{T}$ we define the particles $l$, the last common ancestor of $x$ and $z$, $c$ the first particle which is an ancestor of precisely one of $x$ and $z$, and $p$ the parent of $y$. Note that $c$ must be a child of $l$. Similar to the case of $\mathcal{T}_a$, we note that we cannot have $c=y$. In order to see this, we assume that $c=y$. As $c$ is the first particle to be an ancestor of precisely one of $x$ and $z$, but $z$ is older than $y$ it follows that $z$ must be an ancestor of $x$, and hence $l=z$. But then, $y=c$ and $z=l$ are located at adjacent vertices, which is a contradiction.

In order to count the elements in $\mathcal{T}$, we need to consider two cases depending on whether $c$ is an ancestor of $x$ or of $z$. In the former case, as $c\neq y$,  $c$ must be an ancestor of $p$ and so the particles $x, y, z, l, c$ and $p$ must be related as illustrated in Graph 2 of Figure \ref{fig:heredity}. Similarly, it is clear that in the latter case, the particles must be related as illustrated in Graph 3 in Figure \ref{fig:heredity}.

We now fix $v, w \in \hc{n}$, $1\leq i, j\leq n$ and time intervals $(s, s+ds]$ and $(s+t, s+t+dt]$ where $0\leq s < s+t < u$, and consider the elements in $\mathcal{T}$ where $l$ is located at $v$, $c$ is located at $v+e_i$, $p$ is located at $w-e_j$, $y$ and $z$ are located at $w$, $c$ is born during $(s, s+ds]$ and $y$ is born during $(s+t, s+t+dt]$. We start by counting the triples where $c$ is an ancestor of $x$. Here, a particle is a potential $l$ if it is located at $v$ at time $s$. For each potential $l$, a particle is a corresponding potential $c$ if it is a child of $l$ born at $v+e_i$ during $(s, s+ds]$. Hence the expected number of pairs of potential $l$:s and $c$:s is $m(v, s)\,ds$. For each pair of a potential $l$ and $c$, we see that if one conditions on the BTP at the time of birth of $c$, the corresponding potential triples $(p, y, x)$ originates from $c$ whereas the potential $z$:s originate from $l$. Hence the potential triples $(p, y, x)$ occur independently of the potential $z$:s. Furthermore, for each pair of a potential $l$ and $c$, we see that the expected number of potential $(p, y, x)$ is $m(w-e_j-v-e_i, t)\,dt\,m(\vo-w, u-s-t)$, and the expected number of potential $z$:s is $m(w-v, t)$. Combining this, we see that the expected number of elements in $\mathcal{T}$ corresponding to fixed $v$, $w$, $i$, $j$, fixed time intervals as above and where $c$ is an ancestor of $x$ is
\begin{equation}
m(v, s)\,ds\,m(w-v, t) m(w-e_j-v-e_i, t)\,dt\,m(\vo-w, u-s-t).
\end{equation}
Proceeding in a similar manner for the case where $c$ is an ancestor of $z$ we see that the expected number of corresponding elements in $\mathcal{T}$ is
\begin{equation}
m(v, s)\,ds\,m(w-e_j-v, t) m(w-v-e_i, t)\,dt\,m(\vo-w, u-s-t).
\end{equation}
The proposition follows by summing these expressions over all $v, w \in \hc{n}$, all $1\leq i, j\leq n$ and integrating over all $s, t > 0$ such that $s+t<u$.
\end{proof}

\begin{remark}
In the proof of Proposition \ref{prop:ABformulas}, the only crucial property of the underlying graph is that it should not contain loops (if the graph does contain loops our counting argument may miss elements in $\mathcal{T}_a$ and $\mathcal{T}$). Hence this can directly be generalized to any loop-free graph by replacing the sums over $i$ and $j$ by sums over the corresponding neighborhoods.
\end{remark}

\begin{proposition}\label{prop:Alimit}
For $\fpptime=\ln\left(1+\sqrt{2}\right)$, we have $A(\vo, \fpptime) \rightarrow \frac{\fpptime}{\sqrt{2}}$ as $n \rightarrow \infty$.
\end{proposition}
\begin{proof}
By reordering the sums and integrals in \eqref{eq:Avousumsintegralsgalore} we have
\begin{equation}
A(\vo, \fpptime) =  \int_0^\infty \int_0^\infty \mathbbm{1}_{s+t\leq \fpptime} \sum_{v\in \mathbb{Q}_n} m(v, s) m(\vo-v, \fpptime-s-t) \sum_{i=1}^n \sum_{j=1}^n m(e_j-e_i, t)\,ds\,dt.
\end{equation}
Applying Lemmas \ref{lemma:derivatives} and \ref{lemma:conv}, the right-hand side simplifies to
\begin{equation}
\int_0^\infty \int_0^\infty \mathbbm{1}_{s+t\leq \fpptime} m(\vo, \fpptime-t) \frac{d^2}{dt^2} m(\vz, t)\,ds\,dt = \int_0^\fpptime (\fpptime-t) m(\vo, \fpptime-t) \frac{d^2}{dt^2} m(\vz, t)\,dt,
\end{equation}
and by plugging in the analytical formula \eqref{eq:mvthypercube} for $m(v, t)$ we get
\begin{align}
\begin{split}\label{eq:Avoutractable}
A(\vo, \fpptime) &= \int_0^\fpptime (\fpptime-t) \left(\sinh(\fpptime-t)\right)^n \frac{d^2}{dt^2} \left(\cosh t\right)^n \,dt\\
&= \int_0^\fpptime (\fpptime-t) \left(\sinh(\fpptime-t)\right)^n \left( n + n(n-1)\left(\tanh t\right)^2\right)\left(\cosh t\right)^n \,dt\\
&= \int_0^\fpptime (\fpptime-t) \left( n + n(n-1)\left(\tanh t\right)^2\right)e^{n f(t)} \,dt,
\end{split}
\end{align}
where $f(t) \defeq \ln\left( \sinh(\fpptime-t)\cosh t \right)$.

What follows is a textbook application of the Lebesgue dominated convergence theorem. We begin examining the function $f$. The first and second derivatives of $f$ are given by
\begin{align}
f'(t) &= -\coth(\fpptime-t) + \tanh t\\
f''(t) &= -\csch(\fpptime-t)^2 + \sech(t)^2.
\end{align}
As $\sech t \leq 1$ for all $t\in \mathbb{R}$ and $\csch t \geq 1$ for $0<t<\fpptime$, it follows that $f''(t) <0$ for $0<t<\fpptime$. Hence $f$ is concave in this interval, so in particular $f(t) \leq f(0) + f'(0)\,t = -\sqrt{2}\,t$. Furthermore, we have $\tanh t \leq C t$ for some appropriate $C>0$.

Substituting $t$ by $z=nt$ in \eqref{eq:Avoutractable}, we obtain
\begin{equation}
A(\vo, \fpptime) = \int_0^\infty \mathbbm{1}_{z \leq n \fpptime}\left(\fpptime-\frac{z}{n}\right)\left(1 + (n-1) \tanh\left(\frac{z}{n}\right)^2\right) e^{n f\left(\frac{z}{n}\right)}\,dz.
\end{equation}
It is clear that the integrand is bounded for all $n$ by $\fpptime \left(1+ C z^2\right)e^{-\sqrt{2}\,s}$, which is integrable over $[0, \infty)$. Hence, by dominated convergence, it follows that
\begin{equation}
A(\vo, \fpptime) \rightarrow \int_0^\infty \fpptime e^{-\sqrt{2}\,z} \,dz = \frac{\fpptime}{\sqrt{2}} \text{ as } n\rightarrow \infty.
\end{equation}
\end{proof}

\begin{proposition}
For $\fpptime=\ln\left(1+\sqrt{2}\right)$ we have
\begin{equation}
A(\vo, \fpptime)+B(\vo, \fpptime)\rightarrow \frac{\fpptime e^\fpptime}{\sqrt{2}}+ \frac{1}{3-2\sqrt{2}}\text{ as }n\rightarrow \infty.
\end{equation}
Hence, as $n\rightarrow \infty$ we have $B(\vo, \fpptime) \rightarrow \fpptime + \frac{1}{3-2\sqrt{2}}$.
\end{proposition}
\begin{proof}
By reordering the sums in \eqref{eq:Avousumsintegralsgalore} and applying Lemma \ref{lemma:derivatives} we see that $A(\vo, \fpptime)+B(\vo, \fpptime)$ can be expressed as
\begin{equation}
\frac{1}{2} \sum_{v\in \mathbb{Q}_n} \sum_{w\in \mathbb{Q}_n} \int_0^\infty \int_0^\infty \mathbbm{1}_{s+t\leq \fpptime} m(v, s) m(\vo-w, \fpptime-s-t) \frac{d^2}{dt^2} m(w-v, t)^2\,ds\,dt.
\end{equation}
Letting $\Delta = w-v$, this sum can be rewritten as
\begin{equation}
\frac{1}{2} \sum_{v\in \mathbb{Q}_n} \sum_{ \Delta \in \mathbb{Q}_n} \int_0^\infty \int_0^\infty \mathbbm{1}_{s+t\leq \fpptime} m(v, s) m(\vo-\Delta + v, \fpptime-s-t) \frac{d^2}{dt^2} m(\Delta, t)^2\,ds\,dt,
\end{equation}
which by Lemma \ref{lemma:conv} simplifies to
\begin{equation}
\frac{1}{2} \int_0^\fpptime (\fpptime-t) \sum_{ \Delta \in \mathbb{Q}_n} m(\vo-\Delta, \fpptime-t) \frac{d^2}{dt^2} m(\Delta, t)^2\,dt.
\end{equation}
To evaluate the sum in the above integral we use a small trick. Let us replace $\fpptime-t$ in this sum by $z$ which we consider as a variable not depending on $t$. Then
\begin{align*}
\sum_{ \Delta \in \mathbb{Q}_n} m(\vo-\Delta, z) \frac{d^2}{dt^2} m(\Delta, t)^2 &= \frac{\partial^2}{\partial t^2} \sum_{ \Delta \in \mathbb{Q}_n} m(\vo-\Delta, z) m(\Delta, t)^2.
\end{align*}
By grouping all terms with $\abs{\Delta} = k$ we get
\begin{align*}
\sum_{ \Delta \in \mathbb{Q}_n} m(\vo-\Delta, z) m(\Delta, t)^2 &= \sum_{k=0}^n {n \choose k} \left(\sinh z\right)^k \left( \cosh z \right)^{n-k} \left(\sinh t\right)^{2n-2k} \left( \cosh t \right)^{2k}\\
&= \sum_{k=0}^n {n \choose k} \left(\sinh z \left( \cosh t\right)^2 \right)^k \left( \cosh z \left(\sinh t\right)^2 \right)^{n-k}\\
&= \left(\sinh z \left( \cosh t\right)^2 + \cosh z \left(\sinh t\right)^2 \right)^n\\
&= \left(\frac{1}{2} e^{z} \cosh 2t - \frac{1}{2} e^{-z}\right)^n.
\end{align*}
Note that $\frac{1}{2} e^{z} \cosh 2t - \frac{1}{2} e^{-z} > 0$ for any $t, z \geq 0$. Hence
\begin{align*}
&\sum_{ \Delta \in \mathbb{Q}_n} m(\vo-\Delta, z) \frac{d^2}{dt^2} m(\Delta, t)^2= \frac{ \partial^2}{\partial t^2} \left(\frac{1}{2} e^{z} \cosh 2t - \frac{1}{2} e^{-z}\right)^n\\
&\qquad=2n e^{z} \cosh t \left( \frac{1}{2} e^{z} \cosh 2t - \frac{1}{2} e^{-z} \right)^{n-1} + n(n-1) e^{2z} \left(\sinh 2t\right)^2 \left( \frac{1}{2} e^{z} \cosh 2t - \frac{1}{2} e^{-z} \right)^{n-2}.
\end{align*}
Letting
\begin{align}
f(t) &= \ln \left( \frac{1}{2} e^{\fpptime-t} \cosh 2t - \frac{1}{2} e^{-\fpptime+t} \right)\\
g(t) &= 2 e^{\fpptime-t} \cosh t\, e^{- f(t)} \\
h(t) &= e^{2\fpptime -2t} \left(\sinh t\right)^2e^{-2f(t)}
\end{align}
we can write
\begin{equation}\label{eq:Bfgh}
A(\vo, \fpptime) + B(\vo, \fpptime) = \frac{1}{2} \int_0^\fpptime (\fpptime-t) \left( n\, g(t) + n(n-1)\, h(t) \right)e^{n f(t) }\,dt.
\end{equation}

One can check that $f(0)=f(\fpptime)=0$, $f\left(\frac{1}{2}\right) < -\frac{1}{5}$, and that $f$ has derivatives
\begin{equation}
f'(t) = -1 + 2\frac{\sinh 2t - e^{-2\fpptime+2t}}{\cosh 2t - e^{-2\fpptime+2t}}
\end{equation}
and
\begin{equation}
f''(t) = 4\frac{1-2e^{-2\fpptime}}{\left(\cosh 2t - e^{-2\fpptime+2t}\right)^2}.
\end{equation}
Note that $\frac{1}{2} e^{\fpptime-t} \cosh 2t - \frac{1}{2} e^{-\fpptime+t} = \sinh(\fpptime-t) \left(\cosh t\right)^2 + \cosh(\fpptime-t)\left(\sinh t\right)^2 > 0$ for $t\in[0, \fpptime]$. Hence it follows that $f(t)$ is convex. Furthermore, for $0\leq t \leq \fpptime$, $g(t)$ and $h(t)$ are non-negative bounded functions and $h(t) = O\left( t^2 \right)$.

To evaluate the integral in equation \eqref{eq:Bfgh}, we divide it into two integrals, one over the interval $\left[0, \frac{1}{2}\right]$, and one over $\left[\frac{1}{2}, \fpptime\right]$, that is into the two integrals
\begin{align}
\begin{split}
& \int_0^{\frac{1}{2}} (\fpptime-t) \left( n\,g(t)+n(n-1)\,h(t) \right)e^{n f(t)}\,dt\\
&\qquad \overset{z=nt }= \int_0^\infty
\mathbbm{1}_{z\leq \frac{n}{2}}\left( \fpptime-\frac{z}{n} \right) \left( g\left(\frac{z}{n}\right) + (n-1)\,h\left(\frac{z}{n}\right) \right)e^{n f\left(\frac{z}{n}\right)}\,dz
\end{split}
\intertext{and}
\begin{split}
& \int_{\frac{1}{2}}^{\fpptime} (\fpptime-t) \left( n\,g(t)+n(n-1)\,h(t) \right)e^{n f(t)}\,dt\\
&\qquad \overset{z=n(\fpptime-t)}= \int_{0}^\infty \mathbbm{1}_{z\leq {\left(\fpptime-\frac{1}{2}\right)n}} z \left(  \frac{1}{n}\,g\left(\fpptime-\frac{z}{n}\right)+\frac{n-1}{n}\,h\left(\fpptime-\frac{z}{n}\right) \right)e^{n f\left(\fpptime-\frac{z}{n}\right)} \,dz.
\end{split}
\end{align}
Now, using the convexity of $f(t)$ it is a standard calculation to show that the integrands of these expressions are uniformly dominated by $C\left( 1 + t^2 \right)e^{-\lambda t}$ and $C t e^{-\lambda t}$ respectively, for appropriate positive constants $\lambda$ and  $C$. Hence, by the Lebesgue dominated convergence theorem, these integrals converge to
\begin{equation}
\int_0^\infty 2 \fpptime e^{\fpptime + f'(0)z} \, dz = \frac{2 \fpptime e^{\fpptime}}{-f'(0)} = \sqrt{2}\,\fpptime e^\fpptime
\end{equation}
and
\begin{equation}
\int_0^\infty z \left(\sinh 2\fpptime\right)^2 e^{-f'(\fpptime)z}\,dz = \frac{8}{ f'(\fpptime)^2} = \frac{2}{3-2\sqrt{2}}
\end{equation}
respectively, as $n\rightarrow \infty$. We conclude that
\begin{equation}
A(\vo, \fpptime)+B(\vo, \fpptime) \rightarrow \frac{1}{2}\left( \sqrt{2}\,\fpptime e^\fpptime + \frac{2}{3-2\sqrt{2}}\right) \text{ as } n\rightarrow\infty.
\end{equation}
\end{proof}

\section{Proof of Theorem \ref{thm:unorientedFPP} }\label{sec:bootstrap}
In order to bound $\|T_n-\fpptime\|_p$ it is natural to treat the problems of bounding $T_n-\fpptime$ from above and below separately. To this end, we let $T_n^+$ and $T_n^-$ denote the positive and negative part of $T_n-\fpptime$ respectively, that is, $T_n^+$ is the maximum of $T_n-\fpptime$ and $0$ and $T_n^-$ is the maximum of $\fpptime-T_n$ and $0$. Hence, we can bound $\|T_n-\fpptime\|_p$ by $\|T_n^+\|_p+\|T_n^-\|_p$. We will begin by proving two simple propositions. The first shows that the variance of $T_n$ and the $L^p$-norm of $T_n-\fpptime$ for any $1\leq p < \infty$ are $\Omega\left(\frac{1}{n}\right)$. The second proposition uses the lower bound on $T_n$ obtained by Durrett to prove that $\|T_n^-\|_p=O\left(\frac{1}{n}\right)$. The remaining part of the section will be dedicated to bounding $\|T_n^+\|_p$.

\begin{proposition}
$T_n$ has fluctuations of order at least $\frac{1}{n}$.
\end{proposition}
\begin{proof}
We can write $T_n$ in terms of Richardson's model as the time until the first neighbor of $\vz$ gets infected plus the time from this event until $\vo$ gets infected. It is easy to see that these are independent, and the former is exponentially distributed with mean $\frac{1}{n}$.
\end{proof}
\begin{proposition}\label{prop:Tnminusnorm}
Let $1\leq p < \infty$ be fixed. Then $\|T_n^-\|_p=O\left(\frac{1}{n}\right)$.
\end{proposition}
\begin{proof}
We have
\begin{equation}
\mathbb{E}\left[(T_n^-)^p\right] = \mathbb{E} \int_0^\infty \mathbbm{1}_{t \leq T_n^- }\, p\, t^{p-1}\,dt = \int_0^\infty p\, t^{p-1} \mathbb{P}\left(T_n  \leq \fpptime-t\right)\,dt.
\end{equation}
To bound this, we use that $\mathbb{P}\left(T_n  \leq \fpptime-t\right) \leq m(\vo, \fpptime-t)=\left(\sinh(\fpptime-t)\right)^n$ for any $t\leq \fpptime$ and $\mathbb{P}\left(T_n  \leq \fpptime-t\right) =0$ for $t>\fpptime$ (naturally $T_n$ is always non-negative). It is straightforward to show that $\ln \sinh(\fpptime-t) \leq - \sqrt{2}\,t$ for any $0\leq t \leq \fpptime$. Using this, we conclude that
\begin{equation}
\mathbb{E}\left[(T_n^-)^p\right] \leq \int_0^\fpptime p t^{p-1} e^{-\sqrt{2}\, n t}\,dt = O\left(\frac{1}{n^p}\right).
\end{equation}
\end{proof}

We now turn to the upper bound on $T_n$. Assume $n\geq 4$. Let $\{W_e\}_{e\in E\left(\hc{n}\right)}$ be a collection of independent exponentially distributed random variables with expected value $1$, denoting the passage times of the edges in $\hc{n}$. For any vertex $v$ adjacent to $\vz$ we will use $W_v$ to denote the passage time of the edge between $\vz$ and $v$. Similarly, for any $v$ adjacent to $\vo$, $W_v$ denotes the passage time of the edge between $v$ and $\vo$.

Condition on the weights of all edges connected to either $\vz$ or $\vo$. We pick vertices $a_1$ and $a_2$ adjacent to $\vz$ such that $W_{a_1}$ and $W_{a_2}$ have the smallest and second smallest edge weights respectively among all edges adjacent to $\vz$. Among all $n-2$ neighboring vertices of $\vo$ which are not antipodal to $a_1$ or $a_2$ we then pick $b_1$ and $b_2$ such that $W_{b_1}$ and $W_{b_2}$ have the smallest and second smallest values. Then $W_{a_1}$, $W_{a_2}-W_{a_1}$, $W_{b_1}$ and $W_{b_2}-W_{b_1}$ are independent exponentially distributed random variables with respective expected values $\frac{1}{n}$, $\frac{1}{n-1}$, $\frac{1}{n-2}$ and $\frac{1}{n-3}$.

As $a_1$ and $a_2$ are adjacent to $\vz$ and $b_1$ and $b_2$ are adjacent to $\vo$, there is exactly one coordinate in each of $a_1$ and $a_2$ which is $1$, and exactly one coordinate in $b_1$ and $b_2$ which is $0$. Let the locations of these coordinates in $a_1$, $a_2$, $b_1$ and $b_2$ be denoted by $i$, $j$, $k$ and $l$ respectively. Note that the requirement on $a_1$, $a_2$, $b_1$ and $b_2$ not to be antipodal means that $i$, $j$, $k$ and $l$ are all distinct. We define $H_1$ as the induced subgraph of $\hc{n}$ consisting of all vertices $v\in \hc{n}$ such that the $i$:th coordinate is $1$ and the $k$:th coordinate is $0$. We similarly define $H_2$ as the induced subgraph of $\hc{n}$ consisting of all vertices $v\in \hc{n}$ such that the $j$:th coordinate is $1$ and the $l$:th coordinate is $0$. We furthermore define $H_2'$ as the induced subgraph of $\hc{n}$ whose vertex set is given by $H_2\setminus H_1$. Note that $H_1$ and $H_2'$ are vertex disjoint and hence also edge disjoint.

The idea to bound $T_n$ is essentially to express it in terms of the minimum of the first-passage time from $a_1$ to $b_1$ in $H_1$ and the first-passage time from $a_2$ to $b_2$ in $H_2'$, where the passage times for the edges are taken from $\{W_e\}_{e\in E(\hc{n})}$. As $H_1$ and $H_2$ are both isomorphic to $\hc{n-2}$, where $a_1$ and $b_1$ are antipodal in $H_1$ and $a_2$ and $b_2$ are antipodal in $H_2$, Corollary \ref{cor:Tnnotzero} implies that the corresponding first-passage times in each of $H_1$ and $H_2$ are at most $\fpptime$ with probability bounded away from $0$. However, for our proof it is not needed to make this connection. Rather, we will make use of the slightly stronger statement that the same holds true for $H_2'$. The following proposition is a consequence of Corollary \ref{cor:Tnnotzero}. We postpone the proof of this to the end of the section.

\begin{proposition}\label{prop:H2prime}
There exists a constant $\varepsilon_2 > 0$ such that for all $n\geq 4$, with probability at least $\varepsilon_2$ the first-passage time in $H_2'$ from $a_2$ to $b_2$ is at most $\fpptime$.
\end{proposition}

Now, let $\xi$ denote the indicator function for the event that the first-passage time from $a_2$ to $b_2$ in $H_2'$ is at most $\fpptime$. As $H_1$ is isomorphic to $\hc{n-2}$ it is clear that the first-passage time from $a_1$ to $b_1$ in $H_1$ is distributed as $T_{n-2}$, and so we may couple $T_{n-2}$ to $\{W_e\}_{e\in E\left(\hc{n}\right)}$ such that $T_{n-2}$ denotes this quantity. Note that this means that $\xi$ and $T_{n-2}$ are independent random variables. With this coupling it is clear that $T_n \leq W_{a_1} + W_{b_1} + T_{n-2}$ as this is the passage time of the path that traverses the edge from $\vz$ to $a_1$, then follows the path to $b_1$ in $H_1$ with minimal passage time and lastly traverses the edge from $b_1$ to $\vo$. Furthermore, if $\xi=1$ we similarly see that $T_n \leq W_{a_2} + W_{b_2} + \fpptime$. Combining these bounds we see that for any $n\geq 4$ we have
\begin{equation}\label{eq:recineq}
T_n \leq \xi\left(W_{a_2}+W_{b_2}+\fpptime\right) + (1-\xi)\left(W_{a_1}+W_{b_1}+T_{n-2}\right).
\end{equation}

We may interpret this inequality as follows. We flip a coin $\xi$. If the coin turns up heads then $T_n$ is bounded by $\fpptime$ plus a small penalty. If the coin turns up heads, then we can bound $T_n$ by a small penalty plus $T_{n-2}$, where $T_{n-2}$ is independent of $\xi$. Assuming $n$ is sufficiently large, we can then repeat this process on $T_{n-2}$ and so on until one coin turns up heads. As each coin toss ends up heads with probability at least $\varepsilon_2>0$, this is likely to occur after $O(1)$ steps. Hence the total penalty before this occurs is likely to be small.

We now employ \eqref{eq:recineq} to bound the $L^p$-norm of $T_n^+$. By subtracting $\fpptime$ and taking the positive part of both sides we get
\begin{equation}\label{eq:recineqtau}
T_n^+ \leq \xi\left(W_{a_2}+W_{b_2}\right) + (1-\xi)\left(W_{a_1}+W_{b_1}+T_{n-2}^+\right).
\end{equation}
As $W_{a_2} \geq W_{a_1}$ and $W_{b_2}\geq W_{b_1}$ we can replace $\xi\left(W_{a_2}+W_{b_2}\right) + (1-\xi)\left(W_{a_1}+W_{b_1}\right)$ in the right-hand side of \eqref{eq:recineqtau} by $W_{a_2}+W_{b_2}$. Taking the $L^p$-norm of both sides we obtain the inequality
\begin{equation}
\|T_n^+\|_{p} \leq \|W_{a_2}+W_{b_2}\|_p + \|(1-\xi)T_{n-2}^+\|_p.
\end{equation}
For each fixed $p$, it is straightforward to show that $\|W_{a_2}+W_{b_2}\|_p = O\left(\frac{1}{n}\right)$. Furthermore, as $\xi$ and $T_{n-2}^+$ are independent we have $\|(1-\xi)T_{n-2}^+\|_p = \|(1-\xi)\|_p\|T_{n-2}^+\|_p \leq (1-\varepsilon_2)^{\frac{1}{p}}\|T_{n-2}^+\|_p$. Hence, for any fixed $p$ we have the inequality
\begin{equation}
\|T_n^+\|_{p} \leq O\left(\frac{1}{n}\right)+ (1-\varepsilon_2)^{\frac{1}{p}}\|T_{n-2}^+\|_p.
\end{equation}
As $(1-\varepsilon_2)^{\frac{1}{p}} < 1$ it follows that we must have $\|T_n^+\|_p = O\left(\frac{1}{n}\right)$. Combining this with the corresponding bound on $\|T_n^-\|_p$ from Proposition \ref{prop:Tnminusnorm}, we have $\|T_n-\fpptime\|_p = O\left(\frac{1}{n}\right)$, as desired.

\qed

\vspace{0.3cm}

It only remains to prove Proposition \ref{prop:H2prime}.

In the following argument, we will identify $H_2$ with $\hc{n-2}$ by simply disregarding the two coordinates of the vertices in $H_2$ which are fixed. Hence we will consider $a_2$ and $b_2$ to be the all zeroes and all ones vertices in $\hc{n-2}$ respectively. When seen in this light, is clear that $H_2'$ is the induced subgraph if $H_2$ consisting of all vertices where either the $i'$:th coordinate is $1$ or the $k'$:th coordinate is $0$ for some $i'\neq k'$.

It makes sense to think of $H_2'$ as half a hypercube. For instance, exactly half of the oriented paths from $a_2$ to $b_2$ in $H_2$ are contained in $H_2'$, namely those that move in direction $i'$ before direction $k'$. Now, the paths from $a_2$ to $b_2$ in $H_2$ which are relevant for the early arrivals in the BTP are extremely unlikely to be oriented, but they are not too far from being oriented either. Our approach to showing Proposition \ref{prop:H2prime} is essentially to show that $H_2'$ is a sufficiently large subset of $H_2$ that when considering a BTP on $H_2$ originating at $a_2$, if there is an uncontested particle at $b_2$ at time $\fpptime$, then with probability bounded away from $0$, its ancestral line is contained in $H_2'$.

In order to show this, we need a property of the BTP which was hinted at briefly in \cite{fp93}. Let $\mathbf{X}$ denote a BTP on $\hc{n}$ originating at $\vz$. For any set of paths $A$ in $\hc{n}$, let $X_t(A)$ denote the expected number of particles in the BTP at time $t$ whose ancestral line follows some path in $A$. Let $\{y(t)\}_{t\geq 0}$ denote a simple random walk on $\hc{n}$ starting at $\vz$ with rate $n$, and for each $t\geq 0$ let $\sigma_t$ denote the path that the random walk has followed up to time $t$.

\begin{lemma}\label{lemma:contorandwalk}
Let $S$ denote the set of paths from $\vz$ to $\vo$ in $\hc{n}$. For any $S' \subseteq S$ and for any $t\geq 0$ we have
\begin{equation}\label{eq:distpaths}
\frac{X_t(S')}{X_t(S)} = \mathbb{P}\left(\mathbf{\sigma}_t \in S' \middle\vert y(t)=\vo\right).
\end{equation}
\end{lemma}
\begin{proof}
Let $\sigma$ be any fixed path from $\vz$ to $\vo$ and let $l$ denote the length of $\sigma$. By applying Lemma \ref{lemma:palmBTP}, we get
\begin{equation}
X_t(\{\sigma\}) = \mathbb{E} \sum_{x \in V_{\sigma}(\mathbf{X})} \mathbbm{1}_{T(\mathbf{X}, x) \leq t} = \int_0^\infty \dots \int_0^\infty \mathbbm{1}_{z_1+\dots+z_l\leq t}\,dz_1\dots\,dz_l = \frac{t^l}{l!},
\end{equation}
where $T(\mathbf{X}, x)$ denotes the birth time of $x$. In comparison, it is straightforward to see that $\mathbb{P}\left(\sigma_t = \sigma \right) = e^{-nt} \frac{t^l}{l!}$. It follows that, for any set of paths $A$, we have $X_t(A) = e^{nt}\mathbb{P}(\sigma_t \in A)$, and so in particular
\begin{equation}
\frac{X_t(S')}{X_t(S)} = \frac{\mathbb{P}(\sigma_t \in S')}{\mathbb{P}(\sigma_t \in S)} = \mathbb{P}\left(\mathbf{\sigma}_t \in S' \middle\vert y(t)=\vo\right),
\end{equation}
as desired.
\end{proof}

\begin{lemma}\label{lemma:shortpaths}
Let $\mathbf{X}$ be a BTP on $\hc{n}$ originating at $\vz$. Then with probability $1-o(1)$, all particles at $\vo$ at time $\fpptime$ have ancestral lines of length $\sqrt{2}\, \fpptime n \pm o(n)$.
\end{lemma}
\begin{proof}
We apply Lemma \ref{lemma:contorandwalk} with $t=\fpptime$. As $X_u(S)=m(\vo, \fpptime)=1$ we see that it suffices to show that the number of steps performed by $\{y(t)\}_{t\geq 0}$ up to time $\fpptime$, conditioned on the event that $y(\fpptime)=\vo$, is concentrated around $\sqrt{2}\, \fpptime n$.

In order to show this, we note that if $y(t) = \left(y_1(t), \dots, y_n(t)\right)$ is a simple random walk on $\hc{n}$ with rate $n$, then each coordinate, $y_i(t)$, is an independent simple random walk on $\{0, 1\}$ with rate one. Hence, conditioned on the event that $y(\fpptime)=\vo$, each coordinate $y_i(t)$ is an independent simple random walk on $\{0, 1\}$ conditioned on the event that $y_i(\fpptime)=1$. It is easy to see that the expected number of steps taken by such a process up to time $\fpptime$ is
\begin{equation}
\frac{e^{-\fpptime} }{e^{-\fpptime} }\frac{\fpptime + \frac{\fpptime^3}{2!} + \frac{\fpptime^5}{4!} + \dots }{\fpptime + \frac{\fpptime^3}{3!} + \frac{\fpptime^5}{5!} + \dots } = \fpptime \coth \fpptime = \sqrt{2}\,\fpptime.
\end{equation}
The lemma follows by the law of large numbers.
\end{proof}

\begin{proof}[Proof of Proposition \ref{prop:H2prime}]
Consider the BTP:s $\mathbf{X}$ and $\mathbf{X}'$ on $H_2$ and $H_2'$ respectively, both originating at $a_2$. We may couple these processes such that $\mathbf{X}'$ consists of all particles in $\mathbf{X}$ whose ancestral lines are contained in $H_2'$. Note that any particle in $\mathbf{X}'$ is uncontested in $\mathbf{X'}$ if it is uncontested in $\mathbf{X}$.

As $H_2$ is graph isomorphic to $\hc{n-2}$, we know from Corollary \ref{cor:Tnnotzero} that, with probability bounded away from zero, there exists an uncontested particle in $\mathbf{X}$ at $b_2$ at time $\fpptime$. Furthermore, by Lemma \ref{lemma:shortpaths} we know that if such a particle exists, then with probability $1-o(1)$ the length of its ancestral line is at most $1.25 (n-2)$.

Let us now condition on the event that there exists an uncontested particle $x$ in $\mathbf{X}$ at $\vo$ at time $\fpptime$ whose ancestral line is of length at most $1.25(n-2)$. As a path from $\vz$ to $\vo$ must traverse edges in each of the $n-2$ directions of $\hc{n-2}$ an odd number of times, this bound on the length of the ancestral line implies that there are at least $\frac{7}{8}(n-2)$ directions in which the path followed by the ancestral line of $x$ only traverses one edge. By the symmetry of the hypercube, the distribution of this path must be invariant under permutation of coordinates. Hence, with probability $\approx \frac{49}{128}$, this path only traverses one edge in direction $i'$ and one in direction $k'$, and traverses the edge in direction $i'$ before that in direction $k'$. Hence with probability bounded away from $0$, this path is contained in $H_2'$.

We conclude that with probability bounded away from zero, there exists an uncontested particle at $\vo$ at time $\fpptime$ in $\mathbf{X}'$. The proposition follows from the fact that Richardson's model stochastically dominates the set of uncontested particles in a BTP.
\end{proof}

\section{Proof of Theorem \ref{thm:sigmamin}}\label{sec:proofofsigmamin}

In the following proof we adopt the notation $X_t(A)$, $\{y(t)\}_{t\geq 0}$ and $\sigma_t$ from the previous section. Hence, $\sigma_n$ in the statement of Theorem \ref{thm:sigmamin} will here be denoted by $\sigma_\fpptime$ conditioned on $y(\fpptime)=\vo$. For any set of paths $A$ in $\hc{n}$ we let $Z_t(A)$ denote the expected number of simple paths in $A$ starting at $\vz$ with passage time at most $t$. As $\Gamma_n$ must be a simple path, it follows from the union bound that for any $c\in \mathbb{R}$ and any set $A$ of paths from $\vz$ to $\vo$ in $\hc{n}$, we have
\begin{equation}
\mathbb{P}\left(\Gamma_n \in A \right) \leq Z_{\fpptime+\frac{c}{n} }(A) + \mathbb{P}\left(T_n \geq \fpptime+\frac{c}{n}\right).
\end{equation}
In order to bound the right-hand side of this expression in terms of $\sigma_{\fpptime}$, we first observe that for any $t \geq 0$ we have
\begin{align*}
Z_t(A) &= \sum_{\substack{\sigma \in A\\ \sigma \text{ simple}} } \int_0^\infty \dots \int_0^\infty \mathbbm{1}_{t_1+\dots+t_{\abs{\sigma}} \leq t} e^{-t_1-\dots-t_{\abs{\sigma}}} \,dt_1\dots \,dt_{\abs{\sigma}}\\
&\leq \sum_{\sigma \in A } \int_0^\infty \dots \int_0^\infty \mathbbm{1}_{t_1+\dots+t_{\abs{\sigma}} \leq t} \,dt_1\dots \,dt_{\abs{\sigma}}\\
&=\sum_{\sigma \in A } \frac{t^{\abs{\sigma}}}{\abs{\sigma}!} =X_t(A).
\end{align*}
Secondly, by the Cauchy-Schwarz inequality
\begin{align*}
X_{\fpptime + \frac{c}{n}}(A) &= \sum_{\sigma\in A} \frac{\fpptime^{\abs{\sigma}} }{\abs{\sigma}!} 1 \cdot \left(1+\frac{c}{\fpptime n}\right)^{\abs{\sigma}} \leq \sqrt{\sum_{\sigma\in A} \frac{\fpptime^{\abs{\sigma}} }{\abs{\sigma}!} } \cdot \sqrt{ \sum_{\sigma\in A} \frac{\fpptime^{\abs{\sigma}} }{\abs{\sigma}!} \left(1+\frac{c}{\fpptime n}\right)^{2\abs{\sigma}}}\\
&= \sqrt{X_\fpptime(A)} \cdot \sqrt{ X_{\fpptime \left(1+\frac{c}{\fpptime n}\right)^2}(A)} \leq \sqrt{X_\fpptime (A)} \cdot \sqrt{m\left(\vo, \fpptime\left(1+\frac{c}{\fpptime n}\right)^2\right)}.
\end{align*}
Note that $m\left(\vo, \fpptime\left(1+\frac{c}{\fpptime n}\right)^2\right)$ is bounded as $n\rightarrow\infty$. It follows from Lemma \ref{lemma:contorandwalk} that
\begin{equation}\label{eq:sigmaminprob}
\mathbb{P}\left(\Gamma_{n} \in A\right) \leq \sqrt{ \mathbb{P}\left( \sigma_{\fpptime} \in A\middle\vert y(\fpptime)=\vo\right) }\sqrt{m\left(\vo, \fpptime\left(1+\frac{c}{\fpptime n}\right)^2\right)}   + \mathbb{P}\left(T_n \geq \fpptime + \frac{c}{n}\right).
\end{equation}

Now, consider any asymptotically almost sure property of $\sigma_{\fpptime}$ conditioned on $y(\fpptime)=\vo$. For each $n \geq 1$ let $A_n$ denote the set of paths from $\vz$ to $\vo$  in $\hc{n}$ that do not have this property. Then, by taking $\limsup$ of both sides in \eqref{eq:sigmaminprob} we get
\begin{equation}
\limsup_{n\rightarrow\infty} \mathbb{P}\left(\Gamma_{n} \in A_n\right) \leq \limsup_{n\rightarrow\infty} \mathbb{P}\left(T_n \geq \fpptime + \frac{c}{n}\right).
\end{equation}
The general case of Theorem \ref{thm:sigmamin} follows from Theorem \ref{thm:unorientedFPP} by letting $c\rightarrow\infty$. For the special case of the length of $\Gamma_n$, see the proof of Lemma \ref{lemma:shortpaths}.
\qed

\section*{Acknowledgements}
I am very grateful to my supervisor Peter Hegarty for his constant support and for numerous insightful discussions and comments during this project. I would also like to thank Jeffrey Steif, Johan Tykesson and Johan Wästlund for valuable input on various aspects of this topic.

\end{document}